\documentclass[a4paper,11pt]{article}

\usepackage{amsmath,amsthm,amssymb,amsfonts}
\usepackage{mathtools} 
\usepackage{bm}
\usepackage{mathrsfs} 
\usepackage{changes} 
\usepackage{url} 
\usepackage[toc,page]{appendix} 
\usepackage{enumerate}
\usepackage{enumitem}

\usepackage{cases}

\usepackage{graphicx} 
\usepackage{longtable} 
\usepackage[all,knot,arc,import,poly]{xy}
\usepackage{newfloat} 
\DeclareFloatingEnvironment[fileext=frm,placement={ht},name=Frame]{algfloat}
\usepackage{caption}
\usepackage{placeins}
\usepackage{tikz}
\usepackage{subfig}
\usepackage{pgfplots}
\usepackage{color}
\usepackage{xcolor} 
\usepackage{array}
\usepackage{booktabs}
\usepackage{grffile}
\usetikzlibrary{patterns}
\newcommand{\R}{\mathbb{R}}

\newcommand{\inner}[2]{\langle #1\,,\,#2\rangle} 
\newcommand{\Inner}[2]{\left\langle #1\,|\,#2\right\rangle}
\newcommand{\norm}[1]{\|{#1}\|}
\newcommand{\normq}[1]{\|{#1}\|^2}
\newcommand{\Normq}[1]{\left \|{#1}\right \|^2}
\newcommand{\Norm}[1]{\left\|{#1}\right\|}

\newcommand{\veps}{{\varepsilon}}

\newcommand{\Ha}{{\mathcal{H}}} 
\newcommand{\HH}{{\mathcal{H}}} 
\newcommand{\HHH}{\boldsymbol{\HH}}

\newcommand{\tos}{\rightrightarrows} 
\newcommand{\wto}{\rightharpoonup}

\newcommand{\comenta}[1]{} 

\newcommand{\set}[1]{\{#1\}}

\newcommand{\seq}[1]{\left(#1\right)}

\newcommand{\lab}[1]{\label{#1}}

\newcommand{\bprop}{\begin{proposition}}
\newcommand{\eprop}{\end{proposition}}
\newcommand{\blemm}{\begin{lemma}}
\newcommand{\elemm}{\end{lemma}}
\newcommand{\bdefi}{\begin{definition}}
\newcommand{\edefi}{\end{definition}}
\newcommand{\btheo}{\begin{theorem}}
\newcommand{\etheo}{\end{theorem}}
\newcommand{\bproo}{\begin{proof}}
\newcommand{\eproo}{\end{proof}}
\newcommand{\brema}{\begin{remark}}
\newcommand{\erema}{\end{remark}}
\newcommand{\bassu}{\begin{assumption}}
\newcommand{\eassu}{\end{assumption}}
\newcommand{\bcoro}{\begin{corollary}}
\newcommand{\ecoro}{\end{corollary}}

\newcommand{\benum}{\begin{enumerate}[label = \emph{(\alph*)}]}
\newcommand{\eenum}{\end{enumerate}}

\newcommand{\mgap}{\vspace{.1in}}

\newcommand{\beq}{\begin{equation}}
\newcommand{\eeq}{\end{equation}}

\DeclareMathOperator{\Dom}{Dom} 

\DeclareMathOperator{\Gra}{Gra}

\newtheorem{theorem}{Theorem}[section] 
\newtheorem{lemma}[theorem]{Lemma}
\newtheorem{corollary}[theorem]{Corollary}
\newtheorem{proposition}[theorem]{Proposition}
\newtheorem{remark}[theorem]{Remark}
\newtheorem{definition}[theorem]{Definition}
\newtheorem{assumption}[theorem]{Assumption}

\usepackage[ruled]{algorithm2e}
\SetKwInput{KwInput}{Input}

\usepackage[T1]{fontenc}
\usepackage[utf8]{inputenc}
\usepackage[portuguese,english]{babel} 

\newcommand{\GG}{{\mathcal{G}}} 
\newcommand{\Sa}{{\mathcal{S}}} 
\newcommand{\wh}[1]{\widehat #1}
\newcommand{\wti}[1]{\widetilde #1}

\makeatletter
\def\widebreve#1{\mathop{\vbox{\m@th\ialign{##\crcr\noalign{\kern\p@}%
\brevefill\crcr\noalign{\kern0.1\p@\nointerlineskip}%
$\hfil\displaystyle{#1}\hfil$\crcr}}}\limits}

\def\brevefill{$\m@th \setbox\z@\hbox{}%
 \hfill\scalebox{0.7}{\rotatebox[origin=c]{90}{(}} \kern4pt $}
\makeatletter

\definechangesauthor[name={M. Marques Alves}, color=blue]{M}
\definechangesauthor[name={Raul}, color=green!50!black]{R}
\definechangesauthor[name={Juan}, color=green!50!black]{J}

\begin{document}

\title{An inexact inertial projective splitting algorithm with strong convergence}

\author{ 
M. Marques Alves
\thanks{ Departamento de Matem\'atica,
Universidade Federal de Santa Catarina,
Florian\'opolis, Brazil, 88040-900 ({\tt maicon.alves@ufsc.br}) (corresponding author).}
\and
J. E. Navarro Caballero
\thanks{ Departamento de Matem\'atica,
Universidade Federal de Santa Catarina,
Florian\'opolis, Brazil, 88040-900 ({\tt juan.caballero@ufsc.br}).}
\and
R. T. Marcavillaca
\thanks{ Centro de Modelamiento Matem\'atico (CNRS UMI2807), Universidad de Chile, Santiago, Chile, 8370458 ({\tt raultm.rt@gmail.com; rtintaya@dim.uchile.cl}). 
}
}
\date{\emph{\small Dedicated to Fabián Flores-Bazán on the occasion of his 60th birthday}}

\maketitle
\pagestyle{plain}

\begin{abstract}
We propose and study a strongly convergent inexact inertial projective splitting (PS) algorithm for finding zeros of composite monotone inclusion problems involving the sum of finitely many maximal monotone operators. Strong convergence of the iterates is ensured by projections onto the intersection of appropriately defined half-spaces, even in the absence of inertial effects. We also establish iteration-complexity results for the proposed PS method, which likewise hold without requiring inertial terms. The algorithm includes two inertial sequences, controlled by parameters satisfying mild conditions, while preserving strong convergence and enabling iteration-complexity analysis. Furthermore, for more structured monotone inclusion problems, we derive two variants of the main algorithm that employ forward-backward and forward-backward-forward steps.
%
%
%
%
%
\end{abstract}

\pagestyle{plain}

\section{Introduction}\label{sec:int}

The projective splitting (PS) algorithm was originally proposed by Eckstein and Svaiter in \cite{ES08, ES09} and has since been studied and extended in various directions, including:
(i) PS variants for solving coupled and more general forms of monotone 
inclusions~\cite{Alg+14, ACS14, BC20, BC22};
(ii) asynchronous and block-iterative implementations~\cite{CE18};
(iii) modifications incorporating inertial effects in the iteration process~\cite{AGM24, CG17, Mac25};
(iv) PS-type algorithms with forward and forward-backward steps for subproblems~\cite{JE21, JE22};
(v) a proximal-Newton-type PS algorithm for solving inclusions with smooth operators~\cite{Alv25}; and
(vi) the analysis of convergence rates and iteration-complexity~\cite{EJ19, MS23}.

\mgap

The PS method is designed to numerically solve multi-term (finite-sum) monotone inclusion problems involving finitely many maximal monotone maps, as well as compositions with bounded linear operators (see problem \eqref{eq:probr} below). The main idea is to embed \eqref{eq:probr} into an enlarged product space, reformulating the monotone inclusion problem as a convex feasibility problem that seeks a point in an extended-solution set (see \eqref{eq:KTS}). Given the feasibility problem, cutting planes are constructed using affine 
functionals (see \eqref{eq:localiz} and \eqref{eq:sep.F} below) derived from 
elements in the graph of (the enlargements of) the maximal monotone operators associated with the original problem \eqref{eq:probr}. 
The new framework is more flexible than existing operator splitting methods and demonstrates promising numerical performance in certain 
applications (see, e.g. \cite{JE22}).

\mgap

It is well known that proximal algorithms typically converge weakly (see, e.g., \cite{Gul91, Roc76}), requiring modifications in the iteration mechanism to achieve strong convergence of iterates (see, e.g., \cite{BC01, BC11}). In the context of modern inexact relative-error proximal algorithms~\cite{MS10, SS99b}, this approach was first introduced in \cite{SS00}, where an inexact proximal-type algorithm with strong convergence guarantees for monotone inclusions was proposed.

\mgap

In this paper, we propose and analyze an inertial inexact PS algorithm with strong convergence guarantees. 
Our new PS variant (see Algorithm \ref{alg:iHPEP} below) maintains essentially all attractive properties of 
previous versions while ensuring strong convergence of iterates to the best approximation of the initial guess in the extended-solution set. To the best of our knowledge, this is the first time in the literature that an inertial inexact PS-type algorithm with strong convergence has been proposed. Furthermore, building on Algorithm \ref{alg:iHPEP}, we introduce two variants that incorporate forward-backward and forward-backward-forward (Tseng) steps for subproblems (see Algorithms \ref{alg:iHPEP-FB} and \ref{alg:iHPEP-Tseng-FB}).

\mgap

The rest of the paper is organized as follows. In Section \ref{sec:gn}, we introduce the general notation used in this paper and present some basic results. In Section \ref{sec:main}, we propose and analyze our main algorithm, namely Algorithm \ref{alg:iHPEP}. In Section \ref{sec:FB-TFB}, we introduce two variants of Algorithm \ref{alg:iHPEP} by incorporating forward-backward and forward-backward-forward (Tseng) steps for subproblems. 
Section \ref{sec:conc} is devoted to concluding remarks.
In Appendix \ref{app:well}, we prove that Algorithm \ref{alg:iHPEP} is well-defined, and in Appendix \ref{app:ar}, we provide some auxiliary results.

\section{General notation and some basic results}
\lab{sec:gn}

\subsection{General notation}

Let $\Ha$ be a Hilbert space with inner product $\inner{ \cdot}{ \cdot}$ and induced norm $\norm{\cdot} = \sqrt{\inner{\cdot}{ \cdot}}$.
The \emph{effective domain} and \emph{graph} of a set-valued map $T\colon \HH\tos \HH$ are defined, respectively, by $\Dom T = \set{x \in \HH \mid T(x) \neq \emptyset}$ and 
$\Gra T = \{(x,u) \in \Ha \times \Ha \mid u \in T(x)\}$.
A set-valued map $T \colon \Ha \tos \Ha$ is said to be \emph{monotone operator} if $\inner{x - y}{u - v} \geq 0$ for all $u \in T(x)$ and $v \in T(y)$, and \emph{maximal monotone} if it is monotone and its graph 
is not properly contained in the graph of any other monotone operator on $\Ha$. 
The inverse of $T \colon \Ha \tos \Ha$ is $T^{-1} \colon \Ha \tos \Ha$, defined at any $x \in \Ha$ by $u \in T^{-1}(x)$ if and only if $x \in T(u)$. 
The operator $\gamma T \colon \Ha \tos \Ha$, for $\gamma > 0$, is defined  by $(\gamma T)x := \gamma T(x):=\{\gamma u \mid u \in T(x) \}$.
The sum of $T \colon \Ha \tos \Ha$ and $S \colon \Ha \tos \Ha$ is defined by $(T+S)(x)=\{ u + v \mid u \in T(x),\; v \in S(x)\}$. 
The \emph{resolvent} of a maximal monotone operator $T \colon \Ha \tos \Ha$ is $J_T:=(T + I)^{-1}$, where $I$ denote the identity operator on $\Ha$. 
For a set-valued map $T \colon \Ha \tos \Ha$ and $\varepsilon \geq 0$, the $\varepsilon$-\emph{enlargement} of $T$, $T^{[\varepsilon]}\colon \Ha \tos \Ha$, is defined as
\begin{align} \lab{def:enlarg}
 T^{[\varepsilon]}(x) = \{u \in \Ha \mid \inner{x - y}{u - v} \geq - \varepsilon \enspace \mbox{ for all } (y, v) \in \Gra T\}.
\end{align}
Whenever necessary, we will identify single-valued maps $F\colon \Dom F \subseteq \HH \to \HH$ with its
set-valued representation $F: \HH \tos \HH$ by $F(x) = \set{F(x)}$.
For a nonempty closed and convex set $ C $, the \emph{orthogonal projection} operator onto $C$ is denoted by $P_C$, i.e., $P_C(x)$ is defined as the unique element in $C$ such that 
$\norm{x - P_C(x)} \leq \norm{x-y}$ for all $y \in C$. 
%
\emph{Strong} and \emph{weak} convergence of sequences are denoted by the symbols $\to$ and $\wto$, respectively. 
The \emph{adjoint} of a bounded linear operator $ B $ on $ \HH $ is denoted by $ B^* $ and its \emph{norm} is
defined by $ \norm{B}\coloneqq \sup_{\norm{x}\leq 1}\,\norm{Bx} $.

\mgap

Now let $\HH_0, \HH_1,\ldots, \HH_{n-1}$ be real Hilbert spaces, let $\HH_n \coloneqq \HH_0$ and let $\inner{\cdot}{\cdot}$ and
 $\|\cdot\| = \sqrt{\inner{\cdot}{\cdot}}$ denote the inner product and norm (respectively) in 
 $\HH_i$ ($1\leq i\leq n$). Also, let $\HHH \coloneqq \HH_0\times \HH_1\times \dots\times \HH_{n-1}$ be endowed with the inner product and norm
defined, respectively, as follows (for some $\gamma >0$):
\begin{align} \label{eq:def.inner}
\inner{p}{p'}_{\gamma} = \gamma\inner{z}{z'}+\sum_{i=1}^{n-1}\inner{w_i}{w'_i},\quad
\norm{p}_\gamma = \sqrt{\inner{p}{p}_\gamma},
\end{align}
where $p = (z, w)$, $p' = (z', w')$, with $z, z'\in \HH_0$ and 
$w \coloneqq (w_1,\dots, w_{n-1}), w' \coloneqq (w'_1,\dots, w'_{n-1})\in \HH_1\times \ldots\times \HH_{n-1}$.
Elements in $\HHH$ will be denoted either by $p=(z, w_1,\dots, w_{n-1})$ or simply $p=(z,w)$, where $w=(w_1,\dots, w_{n-1})$, and we will use the notation
 \begin{align} \label{eq:def.wn}
 w_n = -\sum_{i=1}^{n-1}\,G_i^* w_i
\end{align}
whenever $p\in \HHH$ and $ G_i \colon \Ha_0 \to  \Ha_i $ is a bounded linear operator ($1\leq i \leq n-1$).

\subsection{Basic results}

In this subsection, we present some basic results needed for this paper. 
Let $\HH$ and $\GG$ denote real Hilbert spaces.
The following proposition
summarizes some properties of the $\varepsilon$-enlargement $T^{[\varepsilon]}$ (for a proof see, e.g.,  \cite{BSS99}).

\bprop \label{pr:teps}
Let $T, S \colon \Ha\tos \Ha$ be set-valued maps.  The following holds:
\begin{enumerate}[label = \emph{(\alph*)}]
\item \lab{pr:teps.seg}  
If $\varepsilon_1 \leq \varepsilon_2$, then
$T^{[\varepsilon_1]}(x)\subseteq T^{[\varepsilon_2]}(x)$ for all $x \in \Ha$.
\item \lab{teps.ter}  $T^{[\varepsilon_1]}(x)+S^{[\varepsilon_2]}(x) \subseteq
(T+S)^{[\varepsilon_1+\varepsilon_2]}(x)$ for all $x \in \Ha$ and
$\varepsilon_1, \varepsilon_2\geq 0$.
\item \lab{teps.qua} $T$ is monotone if and only if $\Gra T  \subseteq \Gra T^{[0]}$.
\item \lab{pr:teps.qui} $T$ is maximal monotone if and only if $T = T^{[0]}$.
\item \lab{pr:teps.sab} 
If $T$ is maximal monotone, the sequence $(y_k,v_k,\varepsilon_k)$ is
such that $v_k\in T^{[\varepsilon_k]}(y_k)$, $y_k \wto y$,
$v_k \to v$ and $\varepsilon_k \to \varepsilon$, then
$v\in T^{[\varepsilon]}(y)$.
\end{enumerate}
\eprop

\mgap

Recall that $F \colon \Dom F\subseteq \HH \to \Ha$ is said to be 
$L^{-1}$-cocoercive if
\[
\inner{F(x)- F(y)}{x - y}\geq \dfrac{1}{L}\normq{F(x)- F(y)}\quad \mbox{for all} \quad x, y\in \Dom F.
\]

\mgap

\begin{lemma}\emph{(\cite[Lemma 2.2]{Sva14})}
\label{lem:trans}
Let $F \colon \Dom F\subseteq \HH \to \Ha$ be $L^{-1}$-cocoercive. 
Then, for every $z, z' \in \Dom F$,
\begin{align*}
F(z) \in F^{[\varepsilon]}(z'), \quad \mbox{with } \varepsilon = \dfrac{L\normq{z'-z}}{4}.
\end{align*}
\end{lemma}

The proof of the following lemma follows from the definition of $\varepsilon$-enlargements of set-valued maps combined with simple algebraic manipulations.

\begin{lemma}\label{lem:prope}
Let $A \colon \HH \tos \HH$ and $B \colon \GG \tos \GG$ be maximal monotone operators  and 
$\varepsilon^a$, $\varepsilon^b \geq 0$. Then, for all $x \in \HH$ and $y \in \GG$,
\begin{equation*}
A^{[\varepsilon^a]}(x) \times B^{[\varepsilon^b]}(y) \subseteq (A \times B)^{[\varepsilon^a + \varepsilon^b]}(x, y)
\end{equation*}
where $A\times B\colon \HH\times \GG \tos \HH\times \GG$ is defined by $(A \times B)(z, v) = A(z)\times B(v)$.
\emph{(}$\HH\times \GG$ endowed with the usual inner product.\emph{)}
\end{lemma}

\mgap

\blemm\label{lm:proj}
Let $C \subseteq \Ha$ be nonempty, closed and convex and 
let $P_C (\cdot)$ be the orthogonal projection onto $C$. 
Then, for all $x, y\in \Ha$ and $z\in C$,
\begin{align}
\label{eq:pne}
 &\Norm{P_C(x) - P_C(y) }^2 \leq \norm{x - y}^2 - \Norm{\left[x - P_C(x)\right] - \left[y - P_C(y)\right]}^2,\\[4mm]
\label{eq:pcr}
&\langle x - P_C(x), z - P_C(x) \rangle \leq 0.
\end{align}
\elemm

The standard inequality 

\begin{align}\label{eq:ineq.norm}
\Normq{\sum_{i=1}^n x_i} \leq n \sum_{i=1}^n \normq{x_i},
\end{align}
that holds for $x_1, x_2, \dots, x_n$ in $\Ha$, will also be needed in this paper.

\section{A strongly convergent inertial projective splitting algorithm} 
\label{sec:main}

Consider the multi-term monotone inclusion problem
\begin{align}\label{eq:probr}
0 \in \sum_{i=1}^n G_i^\ast T_i (G_i z)
\end{align}
where $n\geq 2$, and  the following assumptions hold:
\begin{itemize}
\item[A1.]\lab{assu:A1} 
For each $i = 1,\dots, n$, the operator $T_i \colon \Ha_i \tos \Ha_i$ is (set-valued) maximal
monotone.
\item[A2.] \lab{assu:A2} For each $i = 1,\dots, n$, it holds that $G_i \colon \Ha_0 \to  \Ha_i$  is a bounded linear operator. Moreover, the linear operator $G_n$ is equal to the identity map in $\Ha_0=\Ha_n$, i.e., $G_n \colon z\mapsto z$
for all $z\in \Ha_0$.
\item[A3.] \lab{assu:A3} The solution set of \eqref{eq:probr} is nonempty, i.e., there exists at least one $z\in \Ha_0$ satisfying
the inclusion in \eqref{eq:probr}.
\end{itemize}

In this section we propose and study a strongly convergent (relative-error) inexact  projective splitting (PS) algorithm with inertial effects, namely Algorithm \ref{alg:iHPEP} below, for solving \eqref{eq:probr}. This algorithm is motivated by the projective splitting framework of \cite{ACS14, ES08}, incorporating inertial effects on iteration and relative-error criterion for the inexact solution of subproblems. Moreover, Algorithm \ref{alg:iHPEP} is also specially designed to guarantee strong convergence of iterates, by defining the next iterate as the 
projection of the initial guess $ p^0 $ onto the intersection of two 
half-spaces (see \eqref{def:Hk-Wk} and \eqref{eq:updt01}). We also refer to the several comments and remarks just before or right after Algorithm \ref{alg:iHPEP} for more details. 

Regarding the analysis of Algorithm \ref{alg:iHPEP}, the main results of this section are Theorems \ref{th:sconv}
and \ref{theo:ite-com} on strong convergence and iteration-complexity, respectively. In Subsection \ref{subsec:conv}, we prove the (strong) asymptotic convergence of Algorithm \ref{alg:iHPEP} and in Subsection \ref{subsec:itec} we analyze its iteration-complexity. 

Before introducing the algorithm, we have to define the extended-solution set (linked to Problem \eqref{eq:probr}) and a special family of affine functionals, as in \eqref{eq:KTS} and \eqref{eq:sep.F}, respectively. The main motivation for introducing the extended-solution set $\Sa_e$ and the 
family $\set{\varphi_k(\cdot)}$ is (see, e.g., \cite{ES08}) to reduce the monotone inclusion \eqref{eq:probr} to the convex feasibility problem (in the enlarged space 
$\HHH \coloneqq \HH_0\times \HH_1\times \dots\times \HH_{n-1}$) of finding a point in $\Sa_e$, and subsequently using the affine functionals $\varphi_k(\cdot)$ as cutting planes.

\mgap 

\noindent
{\bf The extended-solution set of \eqref{eq:probr}.} Note first that, since $G_n = I$ (see Assumption A2), then one can write our monotone inclusion problem \eqref{eq:probr} as
\begin{align}\label{eq:probrn}
0\in \sum_{i=1}^{n-1}G_i^*T_iG_i(z) + T_n(z).
\end{align}
The extended-solution set (or generalized Kuhn-Tucker set) for the problem \eqref{eq:probr} is defined 
as 
\begin{align}\label{eq:KTS}
\Sa_e \coloneqq \left\{(z,w_1,\ldots,w_{n-1})\in \boldsymbol{\Ha} \left |\ w_i\in T_i(G_iz),\, i=1,\ldots,n-1,\, -\sum_{i=1}^{n-1}G^*_iw_i\in T_n(z) \right\} \right. \,.
\end{align}
The set $\Sa_e$ is a closed and convex subset of $\HHH$ (see, e.g., \cite{ACS14}) and, in view of Assumption A3, it is
also nonempty. Actually, it is easy to check that $z$ in $\HH_0$ satisfies \eqref{eq:probrn} if and only if 
there exists $(w_1, \dots, w_{n-1})$ in $\HH_1\times \dots\times \HH_{n-1}$ such that
$p \coloneqq (z, w_1, \dots, w_{n-1})$ belongs to $\Sa_e$. Therefore, our problem \eqref{eq:probr} (or, equivalently, 
\eqref{eq:probrn}) can be posed as the convex feasibility problem of finding a point in the closed and convex set 
$\Sa_e$. Motivated by this observation, next we introduce the family $\set{\varphi_k(\cdot)}$ of affine functionals that will be used as cutting planes in our PS framework (see Eq. \eqref{def:Hk-Wk} in Algorithm \ref{alg:iHPEP}).

\mgap 

\noindent
{\bf The family of (affine) separators.} We first observe that if we pick, for each $i = 1, \dots, n$, a pair $(x_i^k, y_i^k)$ in the graph of $T_i^{[\varepsilon_i^k]}$, i.e., if we pick 
$y_i^k \in T^{[\epsilon_i^k]}_i(x_i^k)$, then by the definition \eqref{def:enlarg} of $T_i^{[\varepsilon_i^k]}$ and the inclusions in \eqref{eq:KTS}
we obtain
\begin{align*}
\langle G_i z - x_i^k, w_i - y_i^k \rangle \geq -\varepsilon_i^k\quad \mbox{for all} \quad (z,w_1, \ldots, w_{n-1}) \in \Sa_e,
\end{align*}
where
\begin{align}\label{eq:defWn}
w_n \coloneqq - \sum_{i=1}^{n-1} G_i^\ast w_i.
\end{align}
(Note that $w_n\in T_n(z)$; see \eqref{eq:KTS} and \eqref{eq:defWn}.) Therefore, by summing up the above inequalities, for $i = 1, \dots, n $, we conclude that
\begin{align} \lab{eq:localiz}
\varphi_k(\underbrace{z, w_1, \dots, w_{n-1}}_{p})\leq 0\quad \mbox{for all} \quad p\in \Sa_e,
\end{align}
where $\varphi_k \colon \boldsymbol{\Ha}\to \R$ is defined as 
\begin{align}\label{eq:sep.F}
\varphi_k(p)  =  \sum_{i=1}^n \left(\inner{G_i z -x_i^k}{y_i^k - w_i} - \varepsilon_i^k\right)\quad \mbox{for all} \quad
p = (z, w_1, \dots, w_{n-1})\in \HHH.
\end{align}

Condition \eqref{eq:localiz} ensures that the set $\Sa_e$, to which we want to find a point in, is a subset of the 
``negative'' half-space $H_k \coloneqq \set{p \in \HHH \mid \varphi_k(p)\leq 0}$. Later on we will built another half-space containing $\Sa_e$ (see Eq. \eqref{def:Hk-Wk}) and successive projections onto the intersection of both these half-spaces (see Eq. \eqref{eq:updt01}) will be instrumental for the mechanism of iteration of Algorithm \ref{alg:iHPEP}.

\mgap

The next lemma, which is motivated by~\cite[Lemma 4]{JE22}, summarizes some 
properties of $\varphi_k(\cdot)$ which will be useful in this paper.

\begin{lemma}\label{lem:grad.sep}
Let $\varphi_k(\cdot)$ and $\Sa_e$ be as in \eqref{eq:sep.F} and \eqref{eq:KTS}, respectively.
The following hold:
\begin{enumerate}[label = \emph{(\alph*)}]
\item \lab{lem:grad.sep.seg} 
We have
\begin{align} \lab{eq:car.varphi}
\varphi_k(p) = \left \langle z, \sum_{i=1}^{n-1}\,G_i^* y_i^k + y_n^k\right\rangle + \sum_{i=1}^{n-1}\inner{x_i^k - G_i x_n^k}{w_i} - \sum_{i=1}^n\,\inner{x_i^k}{y_i^k}
- \sum_{i=1}^n \varepsilon_i^k
\end{align}
for all $p = (z, w_1, \dots, w_{n-1})\in \HHH$. In particular, $\varphi_k(\cdot)$  is an affine functional on $\boldsymbol{\Ha}$.
\item \lab{lem:grad.sep.ter}
$\varphi_k(p)\leq 0$ for all $p\in \Sa_e$.
\item \lab{lem:grad.sep.qua} 
The gradient of $\varphi_k(\cdot)$  with respect to the inner product $\inner{\cdot}{\cdot}_{\gamma}$ as in \eqref{eq:def.inner} is
\begin{align}\label{eq:grad.phi_k}
\nabla\varphi_k=\left(\gamma^{-1}\left(\sum_{i=1}^{n-1}G_i^*y_i^k+y_n^k\right),x_1^k - G_1x_n^k, \ldots,
x_{n-1}^k - G_{n-1}x_{n}^k\right)
\end{align}
and
\begin{align}\label{eq:norm.grad.pi_k}
\norm{\nabla\varphi_k}_{\gamma}^2 =
\gamma^{-1}\left\|\sum_{i=1}^{n-1}G_i^*y_i^k + y_n^k\right\|^2 + \sum_{i=1}^{n-1}\norm{x_i^k - G_ix_{n}^k}^2.
\end{align}
\item \lab{lem:grad.sep.qui}
$\nabla \varphi_k = 0$ if and only if $\sum_{i=1}^{n-1}G_i^*y_i^k+y_n^k = 0$ and
$x_i^k = G_i x_n^k$ for all $i = 1, \dots, n-1$.
\item \lab{lem:grad.sep.sex} If $\nabla \varphi_k = 0$, then $(x_n^k, y_1^k, \ldots, y_{n-1}^k)\in \Sa_e$ whenever
one of the following conditions is satisfied:
\begin{itemize}
\item[\emph{(i)}] $\varepsilon_i^k = 0$, for all $i = 1, \dots, n$.
\item[\emph{(ii)}] There exists $\widetilde p = (\widetilde z, \widetilde w_1, \dots, \widetilde w_{n-1})$ such that
$\varphi_k(\widetilde p) \geq 0$.
\end{itemize}
\end{enumerate}
\end{lemma}
\begin{proof}
The proof follows essentially the same outline of the proof of \cite[Lemma 4]{JE22}. Item 
\emph{\ref{lem:grad.sep.seg}} follows from \eqref{eq:sep.F}, \eqref{eq:defWn}, the fact that 
$G_n = I$ (see Assumption A2) and some simple algebra. Item \emph{\ref{lem:grad.sep.ter}} was already proved; see \eqref{eq:localiz}. 
Item \emph{\ref{lem:grad.sep.qua}} follows from \eqref{eq:car.varphi} and the definition of the inner product
$\inner{\cdot}{\cdot}_{\gamma}$ as in \eqref{eq:def.inner}. Item \emph{\ref{lem:grad.sep.qui}} is an obvious consequence of \eqref{eq:grad.phi_k}. To prove item \emph{\ref{lem:grad.sep.sex}}, assume first that 
$\nabla \varphi_k = 0$. Since $y_i^k \in T^{[\epsilon_i^k]}_i(x_i^k)$, for $i = 1, \dots, n$, by item 
\emph{\ref{lem:grad.sep.qui}} we find 
\begin{align*}
 y_i^k\in T_i^{[\varepsilon_i^k]}(G_i x_n^k)\quad (i = 1, \dots, n-1) \quad \mbox{and} \quad
 - \sum_{i=1}^{n-1}\,G_i^* y_i^k = y_n^k \in T_n^{[\varepsilon_n^k]}(x_n^k).
 \end{align*}
Hence, if  $\varepsilon_i^k = 0$, for all $i = 1, \dots, n$, then $(x_n^k, y_1^k, \ldots, y_{n-1}^k)\in \Sa_e$
by the definition of $\Sa_e$ and Proposition \ref{pr:teps}\emph{\ref{pr:teps.qui}}. 
This finishes the proof of (i).
To prove (ii), and consequently finish the proof of the lemma, assume that $\varphi_k(\widetilde p) \geq 0$ for some $\widetilde p = (\widetilde z, \widetilde w_1, \dots, \widetilde w_{n-1})$. 
First note that, by item \emph{\ref{lem:grad.sep.qui}} and some simple algebra, we have 
\begin{align} \lab{eq:lualua}
\sum_{i=1}^n\, \inner{ x_i^k }{ y_i^k } = \Inner{ x_n^k }{ \sum_{i=1}^{n-1}G_i^*y_i^k+y_n^k } = 0. 
\end{align}
Since $ \nabla \varphi_k = 0 $, from \eqref{eq:car.varphi}, \eqref{eq:grad.phi_k} 
and \eqref{eq:lualua}, we obtain $\sum_{i=1}^k\varepsilon_i^k = -\varphi_k(\widetilde p)\leq 0$. So, 
$\varepsilon_i^k = 0$ (for all $i=1, \dots, n$), which gives the desired result as a consequence of (i).
\end{proof}


\mgap

\begin{algorithm}
\caption{A strongly convergent inertial inexact PS algorithm for solving  \eqref{eq:probr}.}
\SetAlgoLined
\KwInput{$(z^{-1}, w^{-1}_1, \ldots, w^{-1}_{n-1})  = (z^{0}, w^0_1, \ldots, w^0_{n-1})\in \boldsymbol{\Ha}$, scalars $0 \leq  \sigma < 1$ and  $\gamma > 0$.}
\For{ $k=0,1,\ldots$}{ 
Choose $\alpha_k \geq 0$, $\beta_k \geq 0$ and let
\begin{align}
\label{eq:ext.01z}
&\widehat z^{\,k} = z^k+\alpha_k(z^k-z^{k-1}), \quad \widetilde{z}^{\, k} = \widehat{z}^{k} + \beta_k (\widehat{z}^{\, k} - z^0 ),\\[3mm]
\label{eq:ext.02z}
&\widehat{w}_i^{\,k} = w_i^{k} + \alpha_k (w_i^k - w_i^{k-1}), \quad  \widetilde{w}^{\, k}_{i} = \widehat{w}^{\,k}_{i} + \beta_k (\widehat{w}^{\,k}_i - w^{0}_i) \quad (i =1, \ldots, n-1), \\[3mm]
\label{eq:ext.03z}
& \widetilde w_n^{\,k} = -\sum_{i=1}^{n-1}G_i^*\widetilde w_i^{\,k}.
\end{align}
For $i = 1, \ldots, n$, choose scalars $\lambda_i^k > 0$ and compute $(x_i^k, y_i^k, \varepsilon_i^k)\in \Ha\times \Ha\times \R_{+}$ such that
\begin{align}
\begin{aligned}\label{eq:proxT_i}
&y_i^k\in T^{[\varepsilon_{i}^{k}]}_i(x_i^k),\\[1mm]
&\normq{\lambda_i^k y_i^k + x_i^k - (G_i \widetilde{z}^{\,k} + \lambda_i^k  \widetilde{w}_i^{\,k})}  + 2\lambda_i^{k}\varepsilon_{i}^{k}\leq \sigma^2\left(\Norm{G_i \widetilde{z}^k-x_i^k}^2+\Norm{\lambda_i^k( \widetilde{w}_i^k-y_i^k)}^2\right).
\end{aligned}
\end{align}
\If{$x_i^k = G_i x_n^k$ \emph{(}$i=1, \ldots, n-1$\emph{)} \emph{and} 
$\displaystyle \sum_{i=1}^{n-1}G_i^*y_i^k  + y_n^k = 0$}
 	{\Return 
 	\begin{align} \lab{eq:return}
 	(\overline z, \overline w_1, \ldots , \overline w_{n-1}) \coloneqq (x_n^k, y_1^k, \dots, y_{n-1}^k).
 	\end{align}}
Let the affine functional $\varphi_k(\cdot)$ be as in \eqref{eq:sep.F} (see also \eqref{eq:car.varphi}),
let
\begin{align}\label{def:pk}
p^{k} \coloneqq (z^{k}, w_1^{k}, \ldots, w_{n-1}^{k}),
\end{align}
set
\begin{align}\label{def:Hk-Wk}
\begin{aligned}
& H_k =\left\{p\in \boldsymbol{\Ha} \mid \varphi_{k}(p)\leq 0\right\},\\
& W_k =\left\{p\in \boldsymbol{\Ha} \mid \inner{p^0-p^k}{p-p^k}_{\gamma}\leq 0\right\},
\end{aligned}
\end{align}
and compute the next iterate
\begin{align}\label{eq:updt01}
(z^{k+1}, w_1^{k+1}, \ldots , w_{n-1}^{k+1}) = P_{H_k \cap W_k}(p^0).
\end{align}
}
\label{alg:iHPEP}
\end{algorithm}

\mgap
\mgap

\noindent
{\bf Additional notation.}  In order to facilitate the analysis of Algorithm \ref{alg:iHPEP}, we have to introduce some additional notation. Similarly to the definition of $\widetilde w_n^k$ as in \eqref{eq:ext.03z}, 
for the sequences $(w_i^k)$ and $(\widehat w_i^k)$ ($1\leq i \leq n$), we define
\begin{align} \lab{def:wwhat}
w_n^{\,k} \coloneqq -\sum_{i=1}^{n-1}G_i^*w_i^{\,k}\quad \mbox{and}\quad  \widehat w_n^{\,k} \coloneqq -\sum_{i=1}^{n-1}G_i^*\widehat w_i^{\,k}.
\end{align}
Also, similarly to the notation for $p^k$ as in \eqref{def:pk},
\begin{align}\label{eq:indetityp}
\wh p^{\,k} \coloneqq (\wh z^{\,k}, \wh w_1^{\,k}, \ldots , \wh w_{n-1}^{\,k}) \quad \mbox{and} \quad
 \wti p^{\,k} \coloneqq  (\wti z^{\,k}, \wti w_1^{\,k}, \ldots , \wti w_{n-1}^{\,k}).
\end{align}
From \eqref{eq:ext.01z}, \eqref{eq:ext.02z}, \eqref{def:pk} and \eqref{eq:indetityp}, we have
\begin{align*}
\wh p^{\,k} = p^k + \alpha_k (p^k - p^{k-1}) \quad \mbox{and} \quad 
\wti p^{\,k} = \wh p^k + \beta_k (\wh p^k - p^0).
\end{align*}
Furthermore, direct substitution of the first identity in \eqref{eq:ext.02z} into the second
identity in \eqref{def:wwhat} and of the second identity in \eqref{eq:ext.02z} into 
\eqref{eq:ext.03z} yields
\begin{align}\label{eq:identityw}
\wh w_n^{\,k} = w^{k}_n + \alpha_k ( w^{k}_n -  w_n^{k-1}) \quad \mbox{and} \quad
\wti w_n^{\,k} = \wh w^{\,k}_n + \beta_k (\wh w^{\,k}_n -  w_n^{0}).
\end{align}
Whenever it is convenient, we will also use the notation
\begin{align} \lab{def:ei}
e_i^k \coloneqq \lambda_i^k y_i^k + x_i^k - (G_i \widetilde{z}^{\,k} + \lambda_i^k  \widetilde{w}_i^{\,k}),
\quad i = 1, \dots, n.
\end{align}
This implies that  the inequality in \eqref{eq:proxT_i} now reads as
\begin{align} \lab{eq:salg}
\normq{e_i^k}  + 2\lambda_i^{k}\varepsilon_{i}^{k}\leq \sigma^2\left(\Norm{G_i \widetilde{z}^k-x_i^k}^2+\Norm{\lambda_i^k( \widetilde{w}_i^k-y_i^k)}^2\right).
\end{align}
Regarding the just introduced notation, we also observe that, since $G_n$ is assumed to be equal the identity operator (see Assumption A2), definition \eqref{eq:ext.03z} gives us
\begin{align}\lab{eq:sumwtz}
\sum_{i=1}^n G_i^*\widetilde w_i^k = 0.
\end{align}

\mgap

\noindent
Next we make some remarks and comments about Algorithm \ref{alg:iHPEP}:

\begin{enumerate}[label = (\roman*)]
\item Algorithm \ref{alg:iHPEP} operates on the augmented space $\HHH = \HH_0\times \HH_1\times \dots\times \HH_{n-1}$; recall that the extended-solution set $\Sa_e$ is a subset of $\HHH$ and our main problem \eqref{eq:probr} is equivalent to the (convex feasibility) problem of finding a point $p = (z, w_1, \dots, w_{n-1})$ in $\Sa_e$. Also, recall that, by 
Lemma \ref{lem:grad.sep}\emph{\ref{lem:grad.sep.ter}}, the set $\Sa_e$ is a subset of the ``negative'' half-space defined by each affine functional $\varphi_k(\cdot)$ (see the definition of $H_k$ as in \eqref{def:Hk-Wk}). This explains the ``projective'' nature of Algorithm \ref{alg:iHPEP}, where the iteration is defined, as in \eqref{eq:updt01}, by successive projections on the intersection of the half-spaces $H_k$ and $W_k$. The main role of the latter half-space is to enforce the strong convergence of 
the sequences generated Algorithm \ref{alg:iHPEP} (see Theorem \ref{th:sconv} below).
\item The extrapolation steps defined in \eqref{eq:ext.01z}--\eqref{eq:ext.03z} promote \emph{inertial effects} on the iterations; these inertial effects are controlled by the parameters $\alpha_k$ and $\beta_k$.
We note that while the first extrapolation represented by $\wh z^k$ and $\wh w_i^k$  ($i = 1, \dots, n-1$) 
corresponds to the more typical type of extrapolation present in many inertial-based algorithms 
(see, e.g., \cite{AA01, Alv+20, AM20, AC20, AP19}), the second type 
$\wti z^k$ and $\wti w_i^k$ ($i=1, \dots, n$) is introduced by considering the particular iteration mechanism of Algorithm 1. 
Conditions on sequences $\seq{\alpha_k}$ and $\seq{\beta_k}$ to ensure the strong convergence and iteration-complexity of Algorithm \ref{alg:iHPEP} will be given in Assumption \ref{assu:ab} 
(see also Assumption \ref{assu:lambda}).
\item  At each iteration, Algorithm \ref{alg:iHPEP} demands the approximate solution of $n$ prox-subproblems, as in \eqref{eq:proxT_i}. These subproblems can be independently solved in parallel and inexactly within a relative-error criterion that is controlled by the parameter $\sigma\in [0, 1)$.  Note that if $\sigma = 0$, then, by 
\eqref{eq:proxT_i},  we have $\varepsilon_i^k = 0$ and 
$\lambda_i^k y_i^k + x_i^k - (G_i \widetilde{z}^{\,k} + \lambda_i^k  \widetilde{w}_i^{\,k}) = 0$  for $i=1, \dots, n$. So, $y_i^k\in T_i(x_i^k)$ and 
$\lambda_i^k y_i^k + x_i^k = G_i \widetilde{z}^{\,k} + \lambda_i^k  \widetilde{w}_i^{\,k}$
for all $i=1, \dots, n$ (recall that $T^{[0]} = T$). This means that, for a given $\lambda_i^k > 0$, the triple 
$(x_i^k, y_i^k, \varepsilon_i^k)$ defined as
\begin{align} \lab{eq:arflo}
x_i^k \coloneqq G_i \widetilde{z}^{\,k} + \lambda_i^k  \widetilde{w}_i^{\,k} - \lambda_i^k y_i^k, \qquad
y_i^k \coloneqq (\lambda_i^k T_i + I)^{-1}(G_i \widetilde{z}^{\,k} + \lambda_i^k  \widetilde{w}_i^{\,k}),
\qquad \varepsilon_i^k \coloneqq 0
\end{align}
always satisfies the error criterion \eqref{eq:proxT_i}. The main limitation of \eqref{eq:arflo} is the computation of the exact resolvent $(\lambda_i^k T_i + I)^{-1}$ of each maximal monotone operator $T_i$, which may be 
expensive from a computational viewpoint. The relative-error criterion \eqref{eq:proxT_i} overcomes this limitation by allowing the computation of inexact solutions $(x_i^k, y_i^k, \varepsilon_i^k)$, which in practical situations can be achieved by some other ``inner solver'', depending on the specific structure of $T_i$ (the inequality in 
\eqref{eq:proxT_i} being used as a stopping criterion for the ``inner solver'');
see~\cite{Alv+20} for a practical implementation of a similar relative-error criterion in the context of the ADMM splitting algorithm).  
We also mention that \eqref{eq:proxT_i} is a generalization of a relative-error criterion for an inertial PS algorithm that appeared in~\cite{AGM24}. 
In Section \ref{sec:FB-TFB}, we will show how Algorithm \ref{alg:iHPEP} can be used as a framework for the design of new operator splitting methods, where a triple $(x_i^k, y_i^k, \varepsilon_i^k)$ satisfying \eqref{eq:proxT_i} can be computed directly by forward-backward/forward-backward-forward (Tseng) steps, under assumptions of cocoercivity/Lipschitz continuity on the operator $T_i$.
\item In Corollary \ref{cor:airflo}, we will show that if Algorithm \ref{alg:iHPEP} 
stops (see \eqref{eq:return}), then it follows that $(\overline z, \overline w_1, \dots, \overline w_{n-1})$ belongs to the extended-solution set $\Sa_e$; subsequently, in order to analyze the convergence of Algorithm \ref{alg:iHPEP}, we will assume that it generates ``infinite'' sequences (see Remark \ref{rem:620} below).
\item The definition of the next iterate $(z^{k+1}, w_1^{k+1}, \dots, w_{n-1}^{k+1})$ as in \eqref{eq:updt01} by the orthogonal projection of $p^0$ onto the intersection $H_k \cap W_k$ of the half-spaces $H_k$ and $W_k$ is one of the main ingredients to guarantee the strong convergence of Algorithm \ref{alg:iHPEP} (note that $H_k$ is a half-space in $\HHH$ by Lemma \ref{lem:grad.sep}\emph{\ref{lem:grad.sep.seg}}). In the context of inexact relative-error proximal-point methods, this idea traces back to~\cite{SS00} and was also recently 
exploited in~\cite{ACM24} (in the context of inertial methods). The computation of the orthogonal projection
in \eqref{eq:updt01} reduces (at most) to the solution of a $2\times 2$ linear 
system (see \cite[pp. 195 and 196]{SS00} and \cite[Propositions 28.18 and 28.19]{BC11}). Note also that the well-definiteness of the next iterate $(z^{k+1}, w_1^{k+1}, \ldots , w_{n-1}^{k+1})$ as in \eqref{eq:updt01} depends on the nonemptyness of the (closed and convex) set 
$H_k \cap W_k$. This will be discussed in more details in Appendix \ref{app:well} below (see Proposition \ref{pro:inclutionset}).
\item  It is worth mentioning here that the implementation as well as the guarantees of convergence of Algorithm \ref{alg:iHPEP} are independent of the computation of norms and inverses of the bounded linear operators $G_i$ $(1\leq i\leq n)$.
\end{enumerate}

\subsection{Convergence analysis of Algorithm \ref{alg:iHPEP}} 
\label{subsec:conv}

In this subsection, we study the strong convergence properties of Algorithm \ref{alg:iHPEP}. The main result is Theorem \ref{th:sconv} below. Other auxiliary results (Lemmas and Propositions) will also be stated and proved.
Assumptions \ref{assu:lambda} and \ref{assu:ab} (on the sequences of proximal
$\seq{\lambda_i^k}$ $(1\leq i\leq n)$ and inertial $\seq{\alpha_k}$
and $\seq{\beta_k}$ parameters, respectively) will be required in order to prove the convergence of 
Algorithm \ref{alg:iHPEP}.

\mgap

We begin with the following technical lemma:

\begin{lemma}\label{lem:varphi00}
Suppose iteration $k\geq 0$ is reached in \emph{Algorithm \ref{alg:iHPEP}} and 
let $\seq{\widetilde p^k}$ and $\varphi_k(\cdot)$ be as in \eqref{eq:indetityp} and \eqref{eq:sep.F}, respectively.
The following holds:
\begin{enumerate}[label = \emph{(\alph*)}]
\item \label{lem:varphi00.seg} 
For all $i=1, \ldots,n$,
\begin{align}\label{ineq:eps_com}
\inner{ G_i \wti z^{\,k} - x_i^k }{ y^k_i - \wti w^{\, k}_i} - \varepsilon^{k}_i
\geq  \frac{1 - \sigma^2}{2 \lambda^k_i} \left ( \normq{G_i \wti z^{\, k} - x^k_i} + \normq{\lambda^k_i (\wti w^{\,k}_i -y^k_i)} \right).
\end{align}
\item \label{lem:varphi00.ter}
We have
\begin{align*}
\varphi_k(\wti p^{\, k}) \geq \dfrac{1 - \sigma^2}{2}\sum_{i=1}^{n} \left( \dfrac{1}{\lambda_i^k}\normq{G_i\wti z^{\, k} - x_i^k } + \lambda_i^k\normq{\wti w^{\,k}_i - y^k_i}  \right)\geq 0.
\end{align*}
\end{enumerate}
\end{lemma}
\begin{proof}
\emph{\ref{lem:varphi00.seg}} Using the identity $\inner{a}{b} = (1/2) (\normq{a} + \normq{b} - \normq{a-b})$ with $a = G_i \wti z^{\,k} - x^k_i$ and $b = \lambda^k_i (y_i^k - \wti w^{\,k}_i)$, we obtain
\begin{align*}
\inner{ G_i \wti z^{\,k} - x_i^k }{ \lambda^{k}_i(y^k_i - \wti w^{\, k}_i)} = \dfrac{1}{2} \Big( \normq{ G_i\wti z^{\, k}- x^k_i } + \normq{\lambda^{k}_i(y^k_i - \wti w^{\, k}_i)}- \normq{ \underbrace{G_i \wti z^{\,k} -x_i^k  - \lambda^{k}_i(y^k_i - \wti w^{\, k}_i)}_{-e^k_i} }\Big),
\end{align*}
where $e_i^k$ is as in \eqref{def:ei}. Subtracting $\varepsilon_i^k$ in both sides of the latter identity and using the error criterion (inequality) in \eqref{eq:proxT_i} we then find (see also \eqref{eq:salg})
\begin{align*}
\notag
\inner{ G_i \wti z^{\,k} - x_i^k }{ y^k_i - \wti w^{\, k}_i} - \varepsilon^{k}_i = & \dfrac{1}{2 \lambda^{k}_i} 
\left( \normq{ G_i \wti z^{\, k} - x_i^k } + \normq{\lambda^{k}_i (y^k_i - \wti w^{\, k}_i)} - 
\left[\normq{ e^k_i } + 2 \lambda^k_i \varepsilon^k_i\right] \right) \\
\geq & \frac{1- \sigma^2}{2 \lambda^k_i} \left ( \normq{G_i \wti z^{\, k} - x^k_i} + \normq{\lambda^k_i (\wti w^{\,k}_i -y^k_i)} \right),
\end{align*}
which finishes the proof of item \emph{\ref{lem:varphi00.seg}}.

\mgap

\emph{\ref{lem:varphi00.ter}} This follows by summing up the inequality \eqref{ineq:eps_com} for $i =1,\ldots,n$ and using the definitions of 
$\varphi_k(\cdot)$ and $\widetilde p^k$ (see Eqs. \eqref{eq:sep.F} and \eqref{eq:indetityp}).
\end{proof}

\mgap

Next we show that if Algorithm \ref{alg:iHPEP} stops, then it stops at a solution of problem \eqref{eq:probr}:

\bcoro \lab{cor:airflo}
If \emph{Algorithm \ref{alg:iHPEP}} returns $(\overline z, \overline w_1, \dots, \overline w_{n-1})$ as in 
\eqref{eq:return}, then $(\overline z, \overline w_1, \dots, \overline w_{n-1})$ belongs to the extended-solution set
$\Sa_e$. In particular, $\overline z$ is a solution of problem \eqref{eq:probr}.
\ecoro
\bproo
The proof is a direct consequence of Lemma \ref{lem:varphi00}\emph{\ref{lem:varphi00.ter}} and
items  \emph{\ref{lem:grad.sep.qui}} and \emph{\ref{lem:grad.sep.sex}}(ii) of 
Lemma \ref{lem:grad.sep}.
\eproo

\mgap

\brema \lab{rem:620}
\emph{
In view of Corollary \ref{cor:airflo} and Proposition \ref{pro:inclutionset} below, from now on in this section we assume that Algorithm \ref{alg:iHPEP}  doesn't stop, in this way generating ``infinite'' sequences. In particular, from now on in this section, by Lemma \ref{lem:grad.sep}\emph{\ref{lem:grad.sep.qui}}, we have
\begin{align} \lab{eq:620}
 \nabla \varphi_k \neq 0 \; \mbox{ for all } \;  k\geq 0.
\end{align}
}
\erema

\mgap

\begin{proposition}\label{pro:inva}
Let $\seq{p^{k}}$ be generated by \emph{Algorithm \ref{alg:iHPEP}} and let $d_{0,\gamma}$ denote the distance of $p^0$ to the extended-solution set $\Sa_e$ of \eqref{eq:probr}. The following statements hold:
\begin{enumerate}[label = \emph{(\alph*)}]
\item \lab{pro:inva.a}For all $k \geq 0$, 
\begin{align*}
\normq{p^{k+1} - p ^0 }_{\gamma}  \geq \normq{p^k - p^0}_{\gamma} + \normq{p^{k+1} - p^k}_{\gamma}.
\end{align*}
\item \lab{pro:inva.b} 
For all $k \geq 0$,
\begin{align*}
\norm{p^k - p^0}_{\gamma} \leq d_{0,\gamma}.
\end{align*}
\item \lab{pro:inva.c} We have
\begin{align*}
\sum_{k=0}^\infty \normq{p^{k+1} - p^k}_{\gamma} \leq d_{0,\gamma}^2.
\end{align*}
\end{enumerate}
\end{proposition}
\begin{proof}
\emph{\ref{pro:inva.a}} Using Eq. \eqref{eq:pne} in Lemma \ref{lm:proj} (in the space $\HHH$ with the inner product as in \eqref{eq:def.inner}) with $C = W_k$, $x = p^{k+1}$ and $y = p^0$ we obtain
\begin{align*}
\normq{P_{W_k}(p^{k+1}) - P_{W_k}(p^0)}_{\gamma} \leq 
\normq{p^{k+1} - p^0}_{\gamma} - \normq{[\,p^{k+1} - P_{W_k}(p^{k+1})\,] - [\,p^0 - P_{W_k}(p^0)\,]}_{\gamma},
\end{align*}
where $W_k$ is as in \eqref{def:Hk-Wk}. Since $p^{k+1} =P_{W_{k}}(p^{k+1})$, because 
$p^{k+1}\in W_k$ (see \eqref{def:pk} and \eqref{eq:updt01}), and 
$p^k =  P_{W_k}(p^0)$ (see the definition of $W_k$), the latter inequality becomes
\begin{align*}
\normq{p^{k+1} - p^k}_{\gamma} \leq \normq{p^{k+1} - p^0}_{\gamma} - \normq{p^k - p^0}_{\gamma},
\end{align*}
which finishes the proof of item (a).

\mgap

\emph{\ref{pro:inva.b}} Assume that $k\geq 1$ (for $k = 0$ the result is trivial). 
From \eqref{def:pk} and \eqref{eq:updt01} we have $p^k = P_{H_{k-1}\cap W_{k-1}}(p^0)$. Now using Proposition \ref{pro:inclutionset} and the definition of the projection
$P_{H_{k-1} \cap W_{k-1}}(p^0)$ we obtain
\begin{align*}
\norm{p^k - p^0}_{\gamma}  = \norm{P_{H_{k-1} \cap W_{k-1}}(p^0) - p^0}_{\gamma} \leq \norm{P_{\Sa_e} (p^0) -p^0}_{\gamma} = d_{0,\gamma}.
\end{align*}

\mgap

\emph{\ref{pro:inva.c}}
This follows directly from items \emph{\ref{pro:inva.a}} and \emph{\ref{pro:inva.b}} combined with a telescopic sum argument.
\end{proof}

\mgap
\mgap

Next we make a boundedness assumption on the sequences of proximal parameters $\seq{\lambda_i^k}$ generated by Algorithm \ref{alg:iHPEP}:

\bassu \lab{assu:lambda}
For all $i =1, \dots, n$, the sequence $\seq{\lambda_i^k}$ generated by \emph{Algorithm \ref{alg:iHPEP}} satisfies 
\begin{align*}
\lambda_i^k\in [\,\underline \lambda\,, \overline \lambda\,] \quad \emph{for all}\quad k\geq 0,
\end{align*}
where $\overline \lambda \geq \underline \lambda >0$.
\eassu

\mgap

\begin{lemma}\label{lem:varphi}
Consider the sequences evolved by \emph{Algorithm \ref{alg:iHPEP}} and let $\seq{\wti p^{\, k}}$ 
be as in \eqref{eq:indetityp}. Also, let $\varphi_k(\cdot)$ be as in \eqref{eq:sep.F}. The following statements hold:
\begin{enumerate}[label = \emph{(\alph*)}]
\item \lab{lem:varphi.seg} For all $k \geq 0$,
\begin{align}\label{varphi}
\varphi_k(\wti p^{\, k}) \geq \frac{(1 - \sigma^2)\min \{\underline{\lambda}, 1/\overline{\lambda}\}}{2} 
\sum_{i=1}^{n} \left ( \normq{G_i\wti z^{\, k} - x_i^k } + \normq{\wti w^{\,k}_i - y^k_i}  \right).
\end{align}
\item \lab{lem:varphi.ter}  For all $k \geq 0$,
\begin{align}\label{ine:gradient}
\varphi_k (\wti p^{\,k}) \geq \dfrac{1}{c} \normq{\nabla \varphi_k}_{\gamma}
\end{align}
where
\begin{align}\label{eq:constantc}
c  \coloneqq n \max_{i =1, \ldots,n}\normq{G_i} \left( \dfrac{4 \max \{1, 1/\gamma\}}{(1-\sigma^2) \min \{\underline{\lambda}, 1/\overline{\lambda}\}}\right).
\end{align}
\end{enumerate}
\end{lemma}
\begin{proof}
\emph{\ref{lem:varphi.seg}} This is a direct consequence of Lemma \ref{lem:varphi00}\emph{\ref{lem:varphi00.ter}} and
Assumption \ref{assu:lambda}.

\mgap

\emph{\ref{lem:varphi.ter}} Note first that, from \eqref{eq:sumwtz}, \eqref{eq:ineq.norm} and the fact that
$\norm{G_i} = \norm{G_i^\ast}$ ($1 \leq i \leq n$), we have

\begin{align}\label{ine:grad01}
\left \| \sum_{i=1}^{n} G_i^\ast y^k_i  \right \|^2 = \left \| \sum_{i=1}^{n} G_i^\ast(\wti w ^{\,k}_i - y_i^k)\right \|^2 \leq n  \max_{i=1, \ldots ,n} \normq{G_i} \sum_{i=1}^n \normq{\wti w^{\,k}_i - y_i^k}.
\end{align}

Likewise (recall that $G_n = I$),
\begin{align}\label{ine:grad02}
\nonumber
\sum_{i=1}^{n-1} \normq{ x_i^k  - G_ix_n^k} = & \sum_{i=1}^{n-1} \normq{x_i^k - G_i \wti z^{\, k} + G_i( \wti z^{\, k} -x_n^k)}\\
\nonumber
\leq & 2 \sum_{i=1}^{n-1} \left(  \normq{G_i \wti z^{\, k} - x_i^k} +  \normq{G_i (\wti z^{\, k} - x_n^k)} \right)\\
\nonumber
\leq & 2\left( \sum_{i=1}^{n-1}  \normq{G_i \wti z^{\, k} - x_i^k} + n \max_{i=1, \ldots , n} \normq{G_i} 
\normq{ \wti z^{\,k} - x_n^k}\right)\\
\leq & 2n  \max_{i=1, \ldots ,n}  \normq{G_i}  \sum_{i=1}^{n} \normq{G_i \wti z^{\, k} - x_i^k}.
\end{align}

Then, from \eqref{eq:norm.grad.pi_k},  \eqref{varphi}, \eqref{ine:grad01} and  \eqref{ine:grad02} we find
\begin{align*}
\normq{\nabla \varphi_k}_{\gamma}  & = \gamma^{-1} \left \| \sum_{i=1}^n G_i^\ast y_i^k\right \|^2 + \sum_{i=1}^{n-1} \Normq{x_i^k - G_i x_n^k} \\
& \leq 2n \max_{i=1, \ldots,n} \normq{G_i} \max\{1, 1/\gamma\} \sum_{i=1}^n \left( \normq{G_i\wti z^{\,k} - x_i^k} +  \normq{\wti w_i^{\,k} -y_i^k} \right)\\
& \leq  n \max_{i=1, \ldots ,n} \normq{G_i} \left( \dfrac{4 \max\{1, 1/\gamma\}}{(1-\sigma^2)\min\{\underline{\lambda}, 
1/\overline{\lambda}\}}\right) \varphi_k(\wti p^{\,k})\\[1mm]
& = c\, \varphi_k(\wti p^{\,k}),
\end{align*}
which is clearly equivalent to \eqref{ine:gradient}.
\end{proof}

\mgap

\blemm  \lab{lem:kyoto}
Consider the sequences generated by \emph{Algorithm \ref{alg:iHPEP}}, let $\varphi_k(\cdot)$ be as in \eqref{eq:sep.F}
and let the constant $c > 0$ be as in \eqref{eq:constantc}. Then, for all $k\geq 0$,
\begin{enumerate}[label = \emph{(\alph*)}]
\item \lab{lem:kyoto.seg}
$\norm{\nabla \varphi_k}_{\gamma} \leq c\norm{p ^{k+1} - \wti p^{\,k}}_\gamma \quad
\mbox{and} \quad \varphi_k (\wti p^{\,k}) \leq c\normq{p^{k+1} - \wti p^{\,k}}_{\gamma}.$
\item \label{lem:kyoto.ter}
$ \sum_{i=1}^n \varepsilon_i^k \leq \Big(\dfrac{\sigma^2}{1-\sigma^2}\Big) c\normq{p^{k+1} -  \wti p^{\,k}}_\gamma.$
\end{enumerate}
\elemm
\bproo
\emph{\ref{lem:kyoto.seg}} 
To prove the first inequality, recall we are assuming that $\nabla \varphi_k \neq 0$ (see \eqref{eq:620}). Then, as a consequence of the definition of $H_k$ as in \eqref{def:Hk-Wk} and Lemma \ref{lem:grad.sep}\emph{\ref{lem:grad.sep.seg}}, we have
\begin{align*}
P_{H_k}(p) = p - \dfrac{\max\{0, \varphi_k(p)\}}{\normq{\nabla \varphi_k}_\gamma} \nabla \varphi_k \quad \mbox{for all} \quad p\in \HHH.
\end{align*}
Now using that $p^{k+1} \in H_{k}$ (see \eqref{def:pk} and \eqref{eq:updt01}) and  Lemma \ref{lem:varphi}\emph{\ref{lem:varphi.ter}}, we obtain
\begin{align}
\nonumber
\norm{p^{k+1} - \wti p^{\,k}}_{\gamma}  & \geq \norm{P_{H_k}(\wti p^{\,^k}) - \wti p^{\,k}}_{\gamma}\\[2mm]
\lab{eq:kyoto}
& = \frac{\varphi_k (\wti p^{\,k})}{\norm{\nabla \varphi_k}}_{\gamma}\\[2mm]
\nonumber
& \geq  c^{-1}\norm{\nabla \varphi_k}_{\gamma},
\end{align}
which finishes the proof of the first inequality. To prove the second one, combine \eqref{eq:kyoto} and the first inequality. 

\emph{\ref{lem:kyoto.ter}} From the inequality in the error criterion \eqref{eq:proxT_i}, we obtain the bound, for all $i = 1, \dots, n$,
\[
 \varepsilon_{i}^{k}\leq \frac{\sigma^2}{2}\left(\frac{1}{\lambda_i^k}\norm{G_i \widetilde{z}^k-x_i^k}^2
 +\lambda_i^k\norm{\widetilde{w}_i^k - y_i^k}^2\right).
\]
Summing up the above bounds for $i = 1, \dots, n$ and 
using Lemma \ref{lem:varphi00}\emph{\ref{lem:varphi00.ter}} we find
\[
\sum_{i=1}^n\varepsilon_i^k\leq \dfrac{\sigma^2}{1 - \sigma^2}\varphi_k(\widetilde p^k).
\]
To finish the proof of item \emph{\ref{lem:kyoto.ter}}, just combine the latter inequality and the second one in item 
\emph{\ref{lem:kyoto.seg}}.
\eproo

\mgap

From now on, we assume that the sequences $\seq{\alpha_k}$ and 
$\seq{\beta_k}$ given in Algorithm \ref{alg:iHPEP} satisfy the following conditions:

\begin{assumption} 
\lab{assu:ab}
\begin{enumerate}[label = \emph{(\alph*)}]
\item \lab{assu:ab.seg} The sequence $\seq{\alpha_k}$ is bounded, i.e.,
\[
 \overline{\alpha} \coloneqq \sup_{k\geq 0}\, \alpha_k < +\infty.
\]
\item \lab{assu:ab.teg} The sequence $\seq{\beta_k}$ belongs to $\ell_2$, i.e.,
\[
  \overline s \coloneqq \sum_{k=0}^\infty\,\beta^2_k < +\infty.
\]
We will also use the notation
\[
 \overline{\beta} \coloneqq \sup_{k\geq 0}\, \beta_k < +\infty.
\]
\end{enumerate}
\end{assumption}

\mgap
\mgap

The following result is fundamental for both the convergence and iteration-complexity analysis of
Algorithm \ref{alg:iHPEP}. The proof, which will be skipped here, follows the same outline of the proof of \cite[Proposition 3.4(b)]{ACM24}.

\mgap

\begin{proposition} \lab{pr:sumfin}
Let $\seq{p^k}$ be generated by \emph{Algorithm \ref{alg:iHPEP}}, let $\seq{\widetilde p^{\,k}}$ be as
in \eqref{eq:indetityp} and let $d_{0,\gamma}$ denote the distance of $p^0$ to the solution set $\Sa_e$ of \eqref{eq:probr}. Also, let $\overline{\alpha}$, $\overline{\beta}$ and $\overline{s}$ be as in \emph{Assumption \ref{assu:ab}}. Then
\begin{align}\label{ineq:bounded}
\sum_{k=0}^\infty\,\norm{p^{k+1} - \widetilde p^{\,k}}_{\gamma}^2 \leq \Omega_{(\alpha, \beta)}\, d_{0,\gamma}^2
\end{align}
where
\begin{align}\lab{eq:def.omega}
\Omega_{(\alpha, \beta)} \coloneqq (1 + \overline \alpha)\Big[ (1 + \overline \beta)\big[ 1 + \overline \alpha (1 + \overline \beta) \big] + \overline s \Big].
\end{align}
\end{proposition}

\mgap
\mgap

\begin{proposition}\label{pro:lim}
Consider the sequences generated by \emph{Algorithm \ref{alg:iHPEP}}. Then the following holds:
\begin{enumerate}[label = \emph{(\alph*)}]
\item \label{pro:lim.seg} $y_n^k + \sum_{i=1}^{n-1} G_i^\ast y_i^k \to 0$ and, for $i=1, \ldots ,n$, we have $x_i^k - G_i x_n^k \to 0$.
\item \label{pro:lim.ter} For $i=1, \ldots ,n$, we have $\varepsilon_i^k \to 0$.
\item \lab{pro:lim.qua} For $i=1, \ldots ,n$, we have $G_i \wti z^{\, k} - x_i^k \to 0$ and $\wti w^{\, k}_i - y_i^k \to 0$.
\item \lab{pro:lim.qui}  For $i=1, \ldots ,n$, we have $G_i \widehat z^{\, k} - x_i^k \to 0$ and $\widehat w^{\, k}_i - y_i^k \to 0$.
\item\lab{pro:lim.sex} For $i=1, \ldots ,n$, we have $G_i z^k - x_i^k \to 0$ and $w_i^k - y_i^k \to 0$.
\end{enumerate}
\end{proposition}
\begin{proof}
\emph{\ref{pro:lim.seg}} This follows from the first 
inequality in Lemma \ref{lem:kyoto}\emph{\ref{lem:kyoto.seg}} combined with \eqref{eq:norm.grad.pi_k} and \eqref{ineq:bounded}(recall that $G_n = I$).

\mgap

\emph{\ref{pro:lim.ter}} This is a direct consequence 
of Lemma \ref{lem:kyoto}\emph{\ref{lem:kyoto.ter}}
 and  \eqref{ineq:bounded}. 

\mgap

\emph{\ref{pro:lim.qua}} This follows from the second inequality in Lemma \ref{lem:kyoto}\emph{\ref{lem:kyoto.seg}}, Proposition \ref{pr:sumfin} and Lemma \ref{lem:varphi}\emph{\ref{lem:varphi.seg}}.

\mgap

\emph{\ref{pro:lim.qui}} Using the triangle inequality and the second identity in \eqref{eq:ext.01z}, we obtain
\begin{align}\label{ineq:n1}
\notag
\norm{G_i \wh z^k - x_i^k} &  \leq \norm{\wti z^{\,k} - \wh z^{\,k}}\norm{G_i} + \norm{G_i \wti z^{\,k} -x_i^k}\\
& = \beta_k \norm{\wh z^{\,k} - z^0 }\norm{G_i} + \norm{G_i\wti z^{\,k} - x_i^k}, \quad i=1, \ldots,n.
\end{align}
By a similar argument, we also have
\begin{align}\label{ineq:n2}
\notag
\norm{\wh w_i^{\,k} - y_i^k} & \leq \norm{ \wti w _i^{\,k}-\wh w_i^{\,k} } + \norm{ \wti w_i^{\,k} - y_i^k}\\
& = \beta_k \norm{ \wh w_i^{\,k} - w_i^0} + \norm{ \wti w_i^{\,k} - y_i^k}, \quad i=1, \ldots , n-1.
\end{align}
Note also that, using the inequality \eqref{eq:ineq.norm}, the second 
identity in \eqref{eq:ext.02z} and \eqref{eq:ext.03z},
\begin{align}\label{ineq:n3}
\notag
\frac{1}{2} \normq{y_n^k - \wh w_n^{\,k}} & \leq \normq{ \wti w_n^{\,k} - \wh w_n^{\,k}} + \normq{y_n^k - \wti w^{\,k}_n}\\
\notag
& = \Normq{\sum_{i=1}^{n-1} G_i^{\ast}( \wti w^{\,k}_i - \wh w^{\,k}_i)} +   \normq{y_n^k - \wti w^{\,k}_n}\\
&  \leq n\,\max_{i=1, \ldots, n-1} \normq{G_i} \beta_k^2 \sum_{i=1}^{n-1} \normq{\widehat w^{\,k}_i - w_i^0} + 
\normq{y_n^k - \wti w^{\,k}_n}.
\end{align}

 Since the sequence $\seq{p^k}$ is bounded (see Proposition \ref{pro:inva}\emph{\ref{pro:inva.b}}), it follows from \eqref{def:pk}
 and \eqref{def:wwhat}  that $\seq{z^k}$ and $\seq{w_i^k}$ are also bounded, for $i=1, \dots, n$. Hence, 
 $\seq{\wh z^k}$, $\seq{\widetilde w_i^k}$ and 
 $\seq{\wh w_i^k}$ ($i=1, \dots, n$) are bounded as 
 well (use \eqref{eq:ext.01z}--\eqref{eq:ext.03z} and the boundedness 
 of $ \seq{\alpha_k} $ and $ \seq{\beta_k} $). Also, from Assumption \ref{assu:ab}\emph{\ref{assu:ab.teg}} we know that $\beta_k \to 0$. Consequently, the desired result follows from \eqref{ineq:n1}--\eqref{ineq:n3} and item 
 \emph{\ref{pro:lim.qua}}.
 
 \mgap

\emph{\ref{pro:lim.sex}} Similarly to the proof of item \emph{\ref{pro:lim.qui}}, we have the inequalities
\begin{align}\label{ineq:n3.5}
\notag
\norm{G_iz^k  - x_i^k} & \leq \norm{z^k - \wh z^{\,k}} \norm{G_i} + \norm{G_i \wh z^{\,k} - x_i^k}\\[2mm]
& = \alpha_k \norm{z^k - z^{k-1}} \norm{G_i} + \norm{ G_i \wh z^{\,k} - x_i^k}, \quad i=1, \dots, n.
\end{align}
and
\begin{align}\label{ineq:n3.6}
\notag
\norm{w_i^k - y_i^k} & \leq \norm{w_i^{\,k} - \wh w_i^k} + \norm{\wh w_i^{\,k} - y_i^k}\\[2mm]
& = \alpha_k\norm{w_i^k - w_i^{k-1}} + \norm{\wh w_i^{\,k} - y_i^k}, \quad i=1, \dots, n, &&\mbox{[by \eqref{eq:ext.02z} and \eqref{eq:identityw}]}.
\end{align}
Hence, using the fact that $z^k - z^{k-1}\to 0$ and $w_i^k - w_i^{k-1}\to 0$, $i=1, \dots, n$ (because of
Proposition \ref{pro:inva}\emph{\ref{pro:inva.c}} combined with Eqs. \eqref{eq:def.inner}, \eqref{def:pk} and \eqref{def:wwhat})
and the limitation of $\seq{\alpha_k}$ we obtain the desired result directly from \eqref{ineq:n3.5}, \eqref{ineq:n3.6}
and item \emph{\ref{pro:lim.qui}}.
\end{proof}

\mgap
\mgap

\begin{lemma}\label{lem:wcp}
Every weak cluster point of the sequence $\seq{p^k}$ as in \eqref{def:pk} belongs to $\Sa_e$.
\end{lemma}
\begin{proof}
Let $\overline p = (\overline z, \overline w_1, \ldots, \overline w_{n-1}) \in \HHH$ be a weak cluster point of the sequence
$\seq{p^k}$ (recall that $\seq{p^k}$ is bounded, 
by Proposition \ref{pro:inva}\emph{\ref{pro:inva.b}}) and let $\seq{p^{k_j}}$ be a subsequence of $\seq{p^k}$ such that $p^{k_j} \wto \overline p$, i.e,
\begin{align}\label{eq:weakcon}
z^{k_j} \wto \overline z\quad \mbox{ and }\quad w_i^{k_j} \wto \overline w_i, \quad i = 1, \ldots ,n - 1.
\end{align}
Using Proposition \ref{pro:lim}\emph{\ref{pro:lim.sex}}, \eqref{eq:weakcon} and the fact that $G_n = I$ 
(see Assumption A2 above), we also obtain
\begin{align*}
x_n^{k_j} \wto \overline z \quad \mbox{and} \quad y_i^{k_j} \wto \overline w_i, \quad i=1, \ldots, n-1.
\end{align*}
In view Algorithm \ref{alg:iHPEP}'s definition we have  $y_i^{k_j}\in T_i^{[\varepsilon_i^{k_j}]}(x_i^{k_j})$
($1\leq i\leq n$), and from 
Proposition \ref{pro:lim} (items \emph{\ref{pro:lim.seg}} and \emph{\ref{pro:lim.ter}}) we also have
$x_i^{k_j} - G_ix_n^{k_j} \to 0$ ($1\leq i\leq n-1$), 
$y_n^{k_j} + \sum_{i=1}^{n-1}\,G_i^*y_i^{k_j} \to 0$
and $ \veps_i^{k_j} \to 0 $.
Hence, from Corollary \ref{cor:proof},
\begin{align*}
\overline w_i \in T_i (G_i \overline z),\quad i = 1, \ldots, n-1,\quad
 -\sum_{i=1}^{n-1}\,G_i^* \overline w_i \in T_n(\overline z),
\end{align*}
which, in view of the definition of $\mathcal{S}_e$ as in \eqref{eq:KTS}, gives exactly the inclusion
$\overline p\in \mathcal{S}_e$, i.e., every weak cluster point of $\seq{p^k}$ belongs to $\mathcal{S}_e$.
\end{proof}

\mgap
\mgap

Now we prove the strong convergence of  Algorithm \ref{alg:iHPEP}:

\mgap
\mgap

\begin{theorem}[Strong convergence of Algorithm \ref{alg:iHPEP}] \label{th:sconv}
Let the sequence $\seq{p^k}$ be generated by \emph{Algorithm \ref{alg:iHPEP}}, as given in \eqref{def:pk}, 
let the extended-solution set $\mathcal{S}_e$ be as in \eqref{eq:KTS} and suppose 
\emph{Assumptions} \emph{A1--A3}, \emph{\ref{assu:lambda}} and \emph{\ref{assu:ab}} hold. Also, let the sequence $\seq{w_n^k}$ be as in \eqref{def:wwhat}.

Then $\seq{p^k}$ converges strongly to $p^* \coloneqq (z^*, w_1^*, \dots, w_{n-1}^*)$, where $p^*$ denotes the orthogonal projection of $p^0$ onto $\mathcal{S}_e$, $\seq{z^k}$ converges 
strongly to \emph{(}the solution of Problem \eqref{eq:probr}\emph{)} $z^*$
and $\seq{w_i^k}$ converges strongly to $w_i^*$ $(1\leq i\leq n)$, where
$w_n^* \coloneqq -\sum_{i=1}^{n-1}\,G_i^* w_i^*$. Moreover, 
for $i =1, \ldots, n$, the sequences $\seq{x_i^k}$ and $\seq{y_i^k}$ converge strongly to 
$G_i z^*$ and $w_i^*$, respectively.
\end{theorem}
\begin{proof}
\noindent
{\bf Step 1.} We first prove that $p^k\to p^*$. To this end, note that in 
view of Proposition \ref{pro:inva}\emph{\ref{pro:inva.b}} and the fact that 
$d_{0, \gamma} = \norm{p^* - p^0}_\gamma$, we have
\begin{align}\label{eq:xkb}
\norm{p^k - p^0}_{\gamma} \leq \norm{p^{\ast} - p^0}_{\gamma} \quad \mbox{for all}\quad k\geq 0.
\end{align}
Thus, $\seq{p^k}$ is bounded and, by Lemma \ref{lem:wcp}, its weak cluster points belong to $\Sa_e$. 
Then, let $\seq{p^{k_j}}$ be a subsequence of $\seq{p^k}$ and $\overline{p}$ be a weak cluster point of
$\seq{p^k}$ belonging to $\mathcal{S}_e$\,:
\begin{align*}
\overline p\in \mathcal{S}_e,\quad p^{k_j} \rightharpoonup \overline{p}.
\end{align*}
Since $\norm{\cdot}_{\gamma}$ is lower semicontinuous in the weak topology on $\boldsymbol{\Ha}$, we obtain
\begin{align}\label{eq:ast}
\norm{\overline{p} - p^0}_{\gamma} \leq \liminf \norm{p^{k_j} - p^0}_{\gamma} \leq \limsup \norm{p^{k_j} - p^0}_{\gamma} \leq \norm{p^{\ast} -p^0}_{\gamma},
\end{align}
where we used \eqref{eq:xkb} in the latter inequality. Since, by definition, 
$p^{\ast} = P_{\Sa_e}(p^0)$ and $\overline{p} \in \Sa_e$, it follows directly from \eqref{eq:ast} that 
$\overline{p} = p^{\ast}$. Thus, $p^k \wto p^*$, and from \eqref{eq:xkb}, we have
\begin{align*}
\norm{p^* - p^0}_{\gamma} \leq \liminf \norm{p^k - p^0}_{\gamma} \leq \limsup \norm{p^k - p^0}_{\gamma} \leq \norm{p^{\ast} - p^0}_{\gamma}.
\end{align*}
Hence, $\norm{p^k - p^0}_{\gamma} \to  \norm{p^{\ast} - p^0}_{\gamma}$. Since $p^k - p^0 \wto p^* - p^0$, it then follows that $p^k - p^0 \to p^* - p^0$, which implies $p^k \to p^*$.

\mgap

{\bf Step 2.} The result $z^k \to z^*$ and $w_i^k \to w_i^*$ ($1\leq i\leq n$) is a direct consequence of the convergence $p^k \to p^*$ and the definitions $p^k = (z^k, w_1^k, \dots, w_{n-1}^k)$, 
$p^* = (z^*, w_1^*, \dots, w_{n-1}^*)$, $w_n^k = -\sum_{i=1}^{n-1}G_i^*w_i^k$ and 
$w_n^* = -\sum_{i=1}^{n-1}\,G_i^* w_i^*$. To finish the proof of the theorem, it remains to prove that
$\seq{x_i^k}$ and $\seq{y_i^k}$ converge strongly to 
$G_i z^*$ and $w_i^*$, respectively, for all $i =1, \dots, n$. To this end, note that this result follows directly from Proposition \ref{pro:lim}\emph{\ref{pro:lim.sex}} and the fact that $z^k \to z^*$ and $w_i^k \to w_i^*$ ($1\leq i\leq n$) .
\end{proof}

\subsection{Iteration-complexity analysis of Algorithm \ref{alg:iHPEP}}
\lab{subsec:itec}

In this subsection we study the iteration-complexity of Algorithm \ref{alg:iHPEP}.
To this purpose, we first introduce the following notion of approximate solution for 
problem \eqref{eq:probr}: $\{x_i\}_{i=1}^n$, $\{y_i\}_{i=1}^n$ and $\{\varepsilon_i\}_{i=1}^n$, 
where $x_i, y_i \in \HH_i$ and $\varepsilon_i \in \R_+$ ($i=1, \ldots ,n$), define an approximate solution of \eqref{eq:probr}  with tolerance $\rho > 0$ if 
\begin{align}\label{errortolerance}
\begin{aligned}
& y_i \in T_i^{[\varepsilon_i]}(x_i), \quad i=1, \ldots,n, \quad \sum_{i=1}^{n} \varepsilon_i \leq \rho, \quad  \Norm{\sum_{i=1}^{n-1}G_i^\ast y_i + y_n} \leq \rho,\\[2mm]
& \norm{x_i - G_i x_n} \leq \rho, \quad i=1, \ldots, n-1.
\end{aligned}
\end{align}
It is easy to check that if we set $\rho = 0$ in \eqref{errortolerance}, then it follows that $x_n$ is a solution of
\eqref{eq:probr}.

Next we present our result on the iteration-complexity of Algorithm \ref{alg:iHPEP}. We wish to estimate the number of iterations that  Algorithm \ref{alg:iHPEP} requires to find $\{x_i\}_{i=1}^n$, $\{y_i\}_{i=1}^n$ and $\{\varepsilon_i\}_{i=1}^n$ satisfying \eqref{errortolerance}, for a given tolerance $\rho > 0$.

\mgap

\begin{theorem}[Pointwise iteration-complexity of Algorithm \ref{alg:iHPEP}]\label{theo:ite-com}
Consider the sequences evolved by \emph{Algorithm \ref{alg:iHPEP}}, let $d_{0,\gamma}$ denote the distance of $p^0$ to the solution set $\Sa_e$ as in \eqref{eq:KTS} and let 
the constant $\Omega_{(\alpha, \beta)} > 0$ be as in \eqref{eq:def.omega}. 
Also, suppose \emph{Assumptions 
A1--A3, \ref{assu:lambda}} and \emph{\ref{assu:ab}} hold.
Then, for all $k \geq 0$, there exists $\ell \in \{0,1, \ldots,k\}$ such that
\begin{subnumcases}{ }
y_i^\ell \in T_i^{[\varepsilon_i^\ell]}(x_i^\ell), \quad i =1, \ldots ,n, & \lab{peryj} \\[3mm]
\Norm{\sum_{i=1}^{n-1}G_i^\ast y_i^\ell + y_n^\ell} \leq  
\dfrac{n \Big(\displaystyle \max_{i =1, \ldots,n}\normq{G_i}\Big) d_{0,\gamma}}{\sqrt{k+1}} 
\left( \dfrac{4 \sqrt{\gamma} \max\{1, 1/\gamma\}\sqrt{\Omega_{(\alpha, \beta)}}}{(1-\sigma^2) \min \{\underline{\lambda}, 1/\overline{\lambda}\}}\right), & \label{des:v_jcom} \\[3mm]
\norm{x_i^\ell - G_i x_n^\ell} \leq \dfrac{n \Big(\displaystyle \max_{i =1, \ldots,n}\normq{G_i}\Big) d_{0,\gamma}}{\sqrt{k+1}} 
\left( \dfrac{4 \max \{1, 1/\gamma\}\sqrt{\Omega_{(\alpha, \beta)}}}{(1-\sigma^2) \min \{\underline{\lambda}, 1/\overline{\lambda}\}}\right),\quad i=1,\ldots, n-1, & \label{des:e_jcom}  \\[3mm]
\sum_{i=1}^n \varepsilon_i^\ell \leq  
\dfrac{n \Big(\displaystyle \max_{i =1, \ldots,n}\normq{G_i}\Big) d_{0,\gamma}^2}{k+1}\left( \dfrac{4\sigma^2 \max \{1, 1/\gamma\}
\Omega_{(\alpha, \beta)}}{(1-\sigma^2)^2 \min \{\underline{\lambda}, 1/\overline{\lambda}\}}\right). \lab{eq:wind}
\end{subnumcases}
\end{theorem}
\bproo
Direct use of the first inequality in item \emph{\ref{lem:kyoto.seg}} and item \emph{\ref{lem:kyoto.ter}}
of Lemma \ref{lem:kyoto} yields, for all $j\geq 0$,
\begin{align} \lab{eq:vacation}
\Delta_j \coloneqq \max\left\{ \dfrac{\normq{\nabla \varphi_j}_\gamma}{c^2}, 
\dfrac{1 - \sigma^2}{c \,\sigma^2}\sum_{i=1}^n\,\varepsilon_i^j\right\}\leq \normq{p^{j+1} - \widetilde p^j}_\gamma
\end{align}
which in turn combined with Proposition \ref{pr:sumfin} gives
\begin{align*}
\sum_{j=0}^k\,\Delta_j \leq \Omega_{(\alpha, \beta)} d_{0,\gamma}^2
\end{align*}
where $\Omega_{(\alpha, \beta)}$ is as in \eqref{eq:def.omega}. So, choosing $\ell \in \set{0, \dots, k}$ such that
$\Delta_\ell = \min_{j=0, \dots, k}\,\Delta_j$ we find
\begin{align}\lab{eq:vacation02}
(k + 1) \Delta_\ell \leq \Omega_{(\alpha, \beta)} d_{0,\gamma}^2.
\end{align}
From \eqref{eq:vacation} and \eqref{eq:vacation02},
\begin{align} \lab{eq:vacation03}
 \normq{\nabla \varphi_\ell}_\gamma\leq \dfrac{c^2\Omega_{(\alpha, \beta)} d_{0,\gamma}^2}{k+1}
 \quad \mbox{and} \quad 
 \sum_{i=1}^n\,\varepsilon_i^\ell\leq \dfrac{c\,\sigma^2\Omega_{(\alpha, \beta)} d_{0,\gamma}^2}{(1 - \sigma^2)(k+1)}.
\end{align}
(Recall that $c$ is defined as in \eqref{eq:constantc}.) Now combining the first inequality in \eqref{eq:vacation03}
and \eqref{eq:norm.grad.pi_k} we also find
\begin{align*}
\gamma^{-1}\left\|\sum_{i=1}^{n-1}G_i^*y_i^\ell + y_n^\ell\right\|^2 + \sum_{i=1}^{n-1}\norm{x_i^\ell - G_ix_{n}^\ell}^2
\leq \dfrac{c^2\Omega_{(\alpha, \beta)} d_{0,\gamma}^2}{k+1},
\end{align*}
which in turn, after some simple algebra, yields \eqref{des:v_jcom} and \eqref{des:e_jcom}. Note also that \eqref{eq:wind}  follows directly from the second inequality in \eqref{eq:vacation03}. To finish the proof of the theorem, note that 
\eqref{peryj} follows from Algorithm \ref{alg:iHPEP}'s definition (see the inclusions in \eqref{eq:proxT_i}).
\eproo

\section{Variants of the strongly convergent inertial projective splitting algorithm with forward-backward
and forward-backward-forward steps} \label{sec:FB-TFB}

In this section we consider a structured version of the monotone inclusion problem \eqref{eq:probr}, where 
each maximal monotone operator $T_i$ has a splitting  structure: $T_i = F_i + B_i$ for all $i =1, \dots, n$. 
More precisely, we consider the inclusion problem 
\begin{align}\label{eq:Smip}
0\in \sum_{i=1}^{n}G_i^*(\underbrace{F_i + B_i}_{T_i})G_i(z),
\end{align}
where $n\geq 2$ and the following holds:
\begin{itemize}
\item [B1.] $F_i \colon \Dom F_i \subseteq \Ha_i \to \Ha_i$ is a single-valued continuous monotone map.
\item [B2.] $B_i \colon \Ha_i \tos \Ha_i$ is a set-valued maximal monotone operator such that 
$\Dom B_i \subseteq C_i \subseteq \Dom F_i$, where $C_i$ is a nonempty closed and convex subset of $\Ha_i$.
\item [B3.] Assumption A2 of Section \ref{sec:main} holds on the bounded linear operators 
$G_i$ ($1\leq i\leq n$).
\item [B4.] The solution set of \eqref{eq:Smip} is nonempty.
\end{itemize}
Under mild assumptions on the functions $f_i(\cdot)$ and $g_i(\cdot)$ ($1\leq i\leq n$), problem \eqref{eq:Smip} encompasses
the finite-sum convex minimization problem
\[
\min_{x \in \Ha_0} \sum_{i=1}^n f_i(G_ix) + g_i(G_ix)
\]
where, for $i=1, \dots, n$, $f_i \colon \Ha_i \to \R\cup\{+\infty\}$ is a proper closed and convex function and 
$g_i \colon \Ha_i \to \mathbb{R}$ is a sufficiently smooth finite convex function. 

\mgap

\brema \lab{rem:adapter}
\emph{
Assumptions B1 and B2 above guarantee that $T_i = F_i + B_i$ ($1\leq i\leq n$) is 
maximal monotone (see, e.g., \cite{MS11}). This combined with B3 and B4 gives us that the conditions on Assumptions A1--A3 of Section \ref{sec:main} are met in the setting of the present section. 
Hence, one can apply the results of Section \ref{sec:main} to problem \eqref{eq:Smip}.} 
\erema

Assuming that the resolvents $(B_i + I)^{-1}$ of $B_i(\cdot)$ ($1\leq i \leq n$) are easy to 
evaluate, in this section we propose two variants of Algorithm \ref{alg:iHPEP} for (numerically) solving \eqref{eq:Smip}, namely:
\begin{itemize}
\item [1.]  One with forward-backward steps for subproblems, assuming
$F_i(\cdot)$ is $L_i^{-1}$-cocoercive on $\Dom F_i$, i.e., for $i = 1, \dots, n$,
\begin{align}\label{ineq:cocoercivity}
\inner{F_i(x) - F_i(y)}{ x - y}  \geq \frac{1}{L_i} \normq{F_i(x) - F_i(y)} \quad \mbox{for all} \quad x, y \in \Dom F_i,
\end{align}
where $L_i >0$. The proposed algorithm appears as Algorithm \ref{alg:iHPEP-FB} below and the main results on its convergence and iteration-complexity are summarized in Theorem \ref{theo:FBComplexity}.
\item [2.] A second variant with forward-backward-forward (Tseng) steps for the case where 
$F_i (\cdot)$ is $L_i$-Lipschitz continuous, i.e., for $i = 1, \dots, n$,
\begin{align}\label{ineq:lipschitz}
\norm{F_i(x) - F_i(y)} \leq L_i \norm{ x - y} \quad \mbox{for all} \quad  x,y \in \Dom F_i,
\end{align}
where $L_i > 0$. The method is presented as Algorithm \ref{alg:iHPEP-Tseng-FB} below, and the main results on its strong convergence and iteration-complexity are summarized in Theorem \ref{theo:TFBComplexity}.
\end{itemize}

\subsection{A variant of the strongly convergent inertial projective splitting algorithm with forward-backward 
steps} \label{subsec:F-B}

In this subsection, assuming that B1--B4 and \eqref{ineq:cocoercivity} hold, we propose 
Algorithm \ref{alg:iHPEP-FB} to numerically solve \eqref{eq:Smip}.

\mgap

\begin{algorithm}
\caption{A variant of the strongly convergent inertial PS algorithm with forward-backward 
steps for solving \eqref{eq:Smip}.}
\SetAlgoLined
\KwInput{$(z^{-1}, w^{-1}_1, \ldots, w^{-1}_{n-1})  = (z^{0}, w^0_1, \ldots, w^0_{n-1})\in \boldsymbol{\Ha}$, 
scalars $0 < \sigma < 1$ and $\gamma > 0$. For $i=1, \dots, n$, set $\lambda_i = 2 \sigma^2/L_i$.}
\For{ $k=0,1,\ldots$}{ 
Choose $\alpha_k \geq 0$, $\beta_k \geq 0$ and let
\begin{align}
\label{eq:ext.01zFB}
&\widehat z^{\,k} = z^k+\alpha_k(z^k-z^{k-1}), \quad \widetilde{z}^{\, k} = \widehat{z}^{k} + \beta_k (\widehat{z}^{\, k} - z^0 ),\\[3mm]
\label{eq:ext.02zFB}
&\widehat{w}_i^{\,k} = w_i^{k} + \alpha_k (w_i^k - w_i^{k-1}), \quad  \widetilde{w}^{\, k}_{i} = \widehat{w}^{\,k}_{i} + \beta_k (\widehat{w}^{\,k}_i - w^{0}_i) \quad (i =1, \ldots, n-1), \\[3mm]
\label{eq:ext.03zFB}
& \widetilde w_n^{\,k} = -\sum_{i=1}^{n-1}G_i^*\widetilde w_i^{\,k}.
\end{align}
For $i=1, \ldots,n$, let $\overline{z}_i^{\,k} \coloneqq P_{C_i}(G_i\wti z^{\,k})$ and compute
\begin{align}
\label{eq:proxFB_i}
\begin{cases}
x_i^k \coloneqq (\lambda_i B_i+I)^{-1}\left(G_i \widetilde{z}^{\,k} + \lambda_i \widetilde w_i^k - \lambda_i F_i(\overline{z}^{\,k}_i)\right),&\\[3mm]
y_i^k \coloneqq \dfrac{1}{\lambda_i}(G_i \wti z^{\,k}_i + \lambda_i \wti w_i^{\,k}-x_i^k).
\end{cases}
\end{align}
\If{$x_i^k = G_i x_n^k$ \emph{(}$i=1, \ldots, n-1$\emph{)} \emph{and} 
$\displaystyle \sum_{i=1}^{n-1}G_i^*y_i^k  + y_n^k = 0$}
 	{\Return 
 	\begin{align} \lab{eq:returnFB}
 	(\overline z, \overline w_1, \ldots , \overline w_{n-1}) \coloneqq (x_n^k, y_1^k, \dots, y_{n-1}^k).
 	\end{align}}
For $i =1, \dots, n$, let 
\begin{align}\label{eq:epsilon1}
\varepsilon_i^k \coloneqq \dfrac{L_i\normq{x_i^k -  \overline{z}^{\,k}_i}}{4}.
\end{align}
Let the affine functional $\varphi_k(\cdot)$ be as in \eqref{eq:sep.F} (see also \eqref{eq:car.varphi}),
let
\begin{align}\label{def:pk3}
p^{k} \coloneqq (z^{k}, w_1^{k}, \ldots, w_{n-1}^{k}),
\end{align}
set
\begin{align}\label{def:3Hk-Wk}
\begin{aligned}
& H_k =\left\{p\in \boldsymbol{\Ha} \mid \varphi_{k}(p)\leq 0\right\},\\
& W_k =\left\{p\in \boldsymbol{\Ha} \mid \inner{p^0-p^k}{p-p^k}_{\gamma}\leq 0\right\},
\end{aligned}
\end{align}
and compute the next iterate
\begin{align}\label{eq:FBupdt01}
(z^{k+1}, w_1^{k+1}, \ldots , w_{n-1}^{k+1}) = P_{H_k \cap W_k}(p^0).
\end{align}
}
\label{alg:iHPEP-FB}
\end{algorithm}

We  now make some comments regarding Algorithm \ref{alg:iHPEP-FB}:

\begin{enumerate}[label = (\roman*)]
\item Algorithm \ref{alg:iHPEP-FB} is essentially Algorithm \ref{alg:iHPEP} for solving the structural inclusion 
\eqref{eq:Smip} where the triple $(x_i^k, y_i^k, \varepsilon_i^k)$ ($1\leq i\leq n$) satisfying the relative-error condition \eqref{eq:proxT_i} is explicitly computed according to \eqref{eq:proxFB_i} and \eqref{eq:epsilon1}; see Proposition \ref{pro:error-FB} below. We mention that the idea of exploiting structural properties in operator splitting methods by forward-backward steps goes back at least to~\cite{LM79, Pas79},
also appearing more recently (see, e.g., \cite{MS10}) in connection with modern inexact relative-error proximal-point algorithms. 
\item Since Algorithm \ref{alg:iHPEP-FB} is a special instance Algorithm \ref{alg:iHPEP}, it follows that almost all comments and remarks just before or right after the latter algorithm also apply to the former 
(see comments/remarks (i)--(vi) following Algorithm \ref{alg:iHPEP}). 
\item In view of the assumption that the resolvents $(B_i + I)^{-1}$ of $B_i(\cdot)$ ($1\leq i\leq n$) are easy to evaluate, we  obtain that the computation of $x_i^k$ as in \eqref{eq:proxFB_i} essentially depends on the evaluation of the linear operator $G_i$ at $\widetilde z^k$ and of the (monotone) operator $F_i(\cdot)$ at 
$\overline z_i^k$. The main reason to introduce the orthogonal projection $\overline z_i^k$ is to recover feasibility of $G_i \widetilde z^k$ with respect to the $\Dom F_i$ of $F_i(\cdot)$ (see Assumption B2). 
If $\Dom F_i$ is equal to the whole space $\HH_i$, then one could simply set $\overline z_i^k = G_i \widetilde z^k$
and no extra projection onto $C_i$ would be required in the implementation of Algorithm \ref{alg:iHPEP-FB}.
\end{enumerate}

Next we prove that Algorithm \ref{alg:iHPEP-FB} is in fact a special instance of Algorithm \ref{alg:iHPEP}:

\begin{proposition}\label{pro:error-FB}
Let $(x_i^k, y_i^k)$ and $\varepsilon_i^k$ \emph{(}$1\leq i\leq n$\emph{)} be as in \eqref{eq:proxFB_i}
and \eqref{eq:epsilon1}, respectively, and let $\widetilde z^k$ and $\widetilde w_i^k$ \emph{(}$1\leq i\leq n$\emph{)} be as in \eqref{eq:ext.01zFB}--\eqref{eq:ext.03zFB}. Also, let $\lambda_i = 2\sigma^2/L_i$
be defined as in the input of \emph{Algorithm \ref{alg:iHPEP-FB}}. The following holds for all $i=1, \ldots ,n$ and  $k \geq 0$:
\begin{align}\label{relationFB}
\begin{aligned}
& y_i^k \in (F_i^{[\varepsilon_i^k]} + B_i)(x_i^k) \subseteq (F_i + B_i)^{[\varepsilon_i^k]}(x_i^k),\\[2mm]
& \lambda_i y_i^k + x_i^k - (G_i \widetilde{z}^{\,k} + \lambda_i  \widetilde{w}_i^{\,k}) = 0, \quad 
2 \lambda_i \varepsilon_i^k \leq \sigma^2 \normq{G_i \wti z^{k} - x_i^k}.
\end{aligned}
\end{align}
In particular, \emph{Algorithm \ref{alg:iHPEP-FB}} is a special instance of \emph{Algorithm \ref{alg:iHPEP}} 
 \emph{(}with $\lambda_i^k \equiv \lambda_i = 2\sigma^2/L_i$ for all $1\leq i\leq n$\emph{)} for solving \eqref{eq:Smip}.
\end{proposition}
\begin{proof}
Using \eqref{eq:proxFB_i}, the definition of the resolvent $(\lambda_i B_i + I)^{-1}$ and some simple algebraic manipulations we find
$ y_i^k \in F_i(\overline z_i^k) + B_i(x_i^k)$,
which in turn combined with Lemma \ref{lem:trans} and the definition of 
$\varepsilon_i^k$ (see \eqref{eq:epsilon1}) gives the first inclusion in \eqref{relationFB}; 
the second one follows from the first and Proposition \ref{pr:teps}(\emph{\ref{teps.ter}} and \emph{\ref{pr:teps.qui}}).
The identity in \eqref{relationFB} follows directly from the definition of $y_i^k$ as in 
\eqref{eq:proxFB_i}. To finish the proof of \eqref{relationFB}, it remains to prove the inequality. To this end, note that \eqref{eq:epsilon1}, the definition of $\lambda_i$ as $\lambda_i = 2\sigma^2/L_i$ and some simple algebra yield
\begin{align*}
2\lambda_i \varepsilon_i^k \leq \sigma^2\normq{x_i^k - \overline z_i^k}\leq \sigma^2 \normq{G_i \wti z^{k} - x_i^k},
\end{align*}
where in the second inequality we used the definition $\overline{z}_i^{\,k} \coloneqq P_{C_i}(G_i\wti z^{\,k})$ and the inclusion $x_i^k\in C_i$ (because $x_i^k\in \Dom B_i\subseteq C_i$). The last statement of the proposition follows from \eqref{relationFB} and the definitions of Algorithms \ref{alg:iHPEP} and \ref{alg:iHPEP-FB}.
\end{proof}

\mgap

\brema \lab{rem:seifu}
\emph{
In view of Proposition \ref{pro:error-FB} and Corollary \ref{cor:airflo}, from now on in this subsection 
we will assume that Algorithm \ref{alg:iHPEP-FB} doesn't stop.
}
\erema

\mgap

\brema \lab{rem:1042}
\emph{
Using the definition of $\lambda_i^k \equiv \lambda_i = 2\sigma^2/L_i$ in Algorithm \ref{alg:iHPEP-FB}, we see that the condition in Assumption \ref{assu:lambda} is satisfied with
\begin{align} \lab{eq:def.up}
 \underline \lambda \coloneqq \dfrac{2\sigma^2}{\overline L}\quad \mbox{and} \quad
 \overline \lambda \coloneqq \dfrac{2\sigma^2}{\underline L},
\end{align}
where
\begin{align} \lab{eq:def.L}
\overline L \coloneqq \displaystyle \max_{i=1,\dots, n} L_i \quad \mbox{and} \quad
\underline L \coloneqq \displaystyle \min_{i=1,\dots, n} L_i.
\end{align}
}
\erema

\mgap

In what follows we summarize the main results on the strong convergence and iteration-complexity of Algorithm \ref{alg:iHPEP-FB}: 

\begin{theorem}[Strong convergence and iteration-complexity of Algorithm \ref{alg:iHPEP-FB}] \label{theo:FBComplexity}
Consider the sequences evolved by \emph{Algorithm \ref{alg:iHPEP-FB}}, let $d_{0,\gamma}$ denote the distance of $p^0$ to extended-solution set $\Sa_e \neq \emptyset$, let $\overline L, \underline L > 0$ be as in \eqref{eq:def.L} and let $\Omega_{(\alpha, \beta)} > 0$ be as in \eqref{eq:def.omega}. Suppose assumptions \emph{B1--B4} and
\emph{\ref{assu:ab}} are valid (see also \emph{Remarks \ref{rem:adapter} and \ref{rem:1042}}), and that the (cocoercivity) condition 
\eqref{ineq:cocoercivity} also holds for $F_i(\cdot)$ ($1\leq i\leq n$).
Then the following statements are true:
\begin{enumerate}[label = \emph{(\alph*)}]
\item \emph{(}Strong convergence\emph{)} All statements of \emph{Theorem \ref{th:sconv}} also hold for \emph{Algorithm \ref{alg:iHPEP-FB}}.
\item \emph{(}Iteration-complexity\emph{)} For all $k \geq 0$, there exists $\ell \in \{0,1, \ldots,k\}$ such that
\begin{align*}
\begin{cases}
y_i^\ell \in (F_i^{[\varepsilon_i^\ell]} + B_i)(x_i^\ell), \quad i =1, \ldots ,n, & \\[3mm]
\Norm{\sum_{i=1}^{n-1}G_i^\ast y_i^\ell + y_n^\ell} \leq  
\dfrac{n \Big(\displaystyle \max_{i =1, \ldots,n}\normq{G_i}\Big) 
\max\left\{\dfrac{\overline L}{2\sigma^2}, \dfrac{2\sigma^2}{\underline L}\right\}d_{0,\gamma}}{\sqrt{k+1}} 
\left( \dfrac{4\sqrt{\gamma} \max\{1, 1/\gamma\}\sqrt{\Omega_{(\alpha, \beta)}}}{1-\sigma^2}\right), 
&\\[4mm]
\norm{x_i^\ell - G_i x_n^\ell} \leq \dfrac{n \Big(\displaystyle \max_{i =1, \ldots,n}\normq{G_i}\Big) 
\max\left\{\dfrac{\overline L}{2\sigma^2}, \dfrac{2\sigma^2}{\underline L}\right\} d_{0,\gamma}}{\sqrt{k+1}} 
\left( \dfrac{4 \max \{1, 1/\gamma\}\sqrt{\Omega_{(\alpha, \beta)}}}{1-\sigma^2}\right), &\\[5mm]
\hspace{10cm} i=1,\ldots, n-1, &\\[3mm]
\sum_{i=1}^n \varepsilon_i^\ell \leq  \dfrac{n \Big(\displaystyle \max_{i =1, \ldots,n}\normq{G_i}\Big) \max\left\{\dfrac{\overline L}{2\sigma^2}, \dfrac{2\sigma^2}
{\underline L}\right\}d_{0,\gamma}^2}{k+1}\left( \dfrac{4\sigma^2 \max \{1, 1/\gamma\}
\Omega_{(\alpha, \beta)}}{(1-\sigma^2)^2}\right).
\end{cases}
\end{align*}
 \end{enumerate}
\end{theorem}
\begin{proof}
The proof follows directly from Proposition \ref{pro:error-FB}, Remarks \ref{rem:adapter}, \ref{rem:seifu} and \ref{rem:1042}, and Theorems \ref{th:sconv}
and \ref{theo:ite-com}.
\end{proof}

\subsection{A variant of the strongly convergent inertial projective splitting algorithm with forward-backward-forward (Tseng) steps}\label{subsec:TF-B}

In this subsection, assuming that B1--B4 and \eqref{ineq:lipschitz} hold, we propose a 
variant of Algorithm \ref{alg:iHPEP} (Algorithm \ref{alg:iHPEP-Tseng-FB} below) tailored for the structured inclusion \eqref{eq:Smip}.

\mgap

\begin{algorithm}
\caption{A variant of the strongly convergent inertial PS algorithm with forward-backward-forward 
steps for solving \eqref{eq:Smip}.}
\SetAlgoLined
\KwInput{$(z^{-1},w^{-1}_1,\ldots,w^{-1}_{n-1})  = (z^0,w^0_1,\ldots,w^0_{n-1})\in \boldsymbol{\Ha}$, scalars 
$0 < \sigma<1$ and $\gamma >0$. For $i = 1, \dots, n$, set $\lambda_i =  \sigma/L_i$.}
\For{ $k=0,1,\ldots$}{ 
Choose $\alpha_k \geq 0$, $\beta_k \geq 0$ and let
\begin{align}
\label{eq:ext.01zTFB}
&\widehat z^{\,k} = z^k+\alpha_k(z^k-z^{k-1}), \quad \widetilde{z}^{\, k} = \widehat{z}^{k} + \beta_k (\widehat{z}^{\, k} - z^0 ),\\[3mm]
\label{eq:ext.02zTFB}
&\widehat{w}_i^{\,k} = w_i^{k} + \alpha_k (w_i^k - w_i^{k-1}), \quad  \widetilde{w}^{\, k}_{i} = \widehat{w}^{\,k}_{i} + \beta_k (\widehat{w}^{\,k}_i - w^{0}_i) \quad (i =1, \ldots, n-1), \\[3mm]
\label{eq:ext.03zTFB}
& \widetilde w_n^{\,k} = -\sum_{i=1}^{n-1}G_i^*\widetilde w_i^{\,k}.
\end{align}
For $i=1, \dots, n$, let $\overline{z}_i^k \coloneqq P_{C_i}(G_i \wti z^{\,k})$ and compute  
\begin{align} \label{eq:proxTFB_i}
\begin{cases}
x_i^k = (\lambda_i B_i+I)^{-1}\left(G_i \widetilde{z}^{\,k} + \lambda_i \wti w_i^k - \lambda_i F_i(\overline{z}_i^k)\right),&\\[3mm]
y_i^k = \dfrac{G_i \wti z^{\,k} + \lambda_i \wti w_i^k -x_i^k}{\lambda_i} + F_i(x_i^k) - F_i(\overline{z}_i^k).
\end{cases}
\end{align}
\If{$x_i^k = G_i x_n^k$ \emph{(}$i=1, \ldots, n-1$\emph{)} \emph{and} 
$\displaystyle \sum_{i=1}^{n-1}G_i^*y_i^k  + y_n^k = 0$}
 	{\Return 
 	\begin{align} \lab{eq:returnTFB}
 	(\overline z, \overline w_1, \ldots , \overline w_{n-1}) \coloneqq (x_n^k, y_1^k, \dots, y_{n-1}^k).
 	\end{align}}
For $i =1, \dots, n$, set
\begin{align}\label{eq:epsilon1T}
\varepsilon_i^k \coloneqq 0.
\end{align}
Let the affine functional $\varphi_k(\cdot)$ be as in \eqref{eq:sep.F} (see also \eqref{eq:car.varphi}),
let
\begin{align}\label{def:Tpk3}
p^{k} \coloneqq (z^{k}, w_1^{k}, \ldots, w_{n-1}^{k}),
\end{align}
set
\begin{align}\label{def:T3Hk-Wk}
\begin{aligned}
& H_k =\left\{p\in \boldsymbol{\Ha} \mid \varphi_{k}(p)\leq 0\right\},\\
& W_k =\left\{p\in \boldsymbol{\Ha} \mid \inner{p^0-p^k}{p-p^k}_{\gamma}\leq 0\right\},
\end{aligned}
\end{align}
and compute the next iterate
\begin{align}\label{eq:Tupdt01}
(z^{k+1}, w_1^{k+1}, \ldots , w_{n-1}^{k+1}) = P_{H_k \cap W_k}(p^0).
\end{align}
}
\label{alg:iHPEP-Tseng-FB}
\end{algorithm}

\mgap 

Next we make some remarks regarding Algorithm \ref{alg:iHPEP-Tseng-FB}:
\begin{enumerate}[label = (\roman*)]
\item As it was the case for Algorithm \ref{alg:iHPEP-FB}, Algorithm \ref{alg:iHPEP-Tseng-FB} is also a special instance of Algorithm \ref{alg:iHPEP} (see Proposition \ref{pro:error-TFB} below). More precisely, Algorithm \ref{alg:iHPEP-Tseng-FB} is essentially 
Algorithm \ref{alg:iHPEP} for solving \eqref{eq:Smip} where the triple $(x_i^k, y_i^k, \varepsilon_i^k)$ ($1\leq i\leq n$) satisfying condition \eqref{eq:proxT_i} is also explicitly computed but now according to \eqref{eq:proxTFB_i}
and \eqref{eq:epsilon1T}. The idea of using forward-backward-forward steps as \eqref{eq:proxTFB_i} was introduced by P. Tseng in the seminal paper~\cite{Tse00} (see also \cite{MS10}).
\item As it was also observed for Algorithm \ref{alg:iHPEP-FB},  since Algorithm \ref{alg:iHPEP-Tseng-FB} is a special instance Algorithm \ref{alg:iHPEP}, it follows that almost all the comments/remarks (i)--(vi) 
just before or right after Algorithm \ref{alg:iHPEP} also apply to Algorithm \ref{alg:iHPEP-Tseng-FB}.
\item Algorithm \ref{alg:iHPEP-Tseng-FB} operates in a more flexible setting, regarding the monotone inclusion problem \eqref{eq:Smip} -- because cocoercivity is more restrictive than Lipschitz continuity (see \eqref{ineq:cocoercivity} and \eqref{ineq:lipschitz}) --, at the cost, when compared to 
\eqref{eq:proxFB_i}, of an extra evaluation of the single-valued operator $F_i(\cdot)$ at $x_i^k$.
\end{enumerate}

\mgap

Algorithm \ref{alg:iHPEP-Tseng-FB} is also a special instance of Algorithm \ref{alg:iHPEP}:

\begin{proposition}\label{pro:error-TFB}
Let $(x_i^k, y_i^k)$ \emph{(}$1\leq i\leq n$\emph{)} be as in \eqref{eq:proxTFB_i} and let $\widetilde z^k$ and $\widetilde w_i^k$ \emph{(}$1\leq i\leq n$\emph{)} be as in \eqref{eq:ext.01zTFB}--\eqref{eq:ext.03zTFB}. Also, let $\lambda_i = \sigma/L_i$
be defined as in the input of \emph{Algorithm \ref{alg:iHPEP-Tseng-FB}}.
The following holds, for all $i=1, \ldots ,n$ and  $k \geq 0$:
\begin{align}\label{relationTFB}
\begin{aligned}
& y_i^k \in  (F_i + B_i)(x_i^k),\\[1mm]
& \norm{\lambda_i y_i^k + x_i^k  - (G_i \wti z^{\,k} + \lambda_i \wti w_i^k)} \leq 
\sigma\norm{G_i \wti z^{\,k}_i - x_i^k}.
\end{aligned}
\end{align}
In particular, \emph{Algorithm \ref{alg:iHPEP-Tseng-FB}} is special instance of \emph{Algorithm \ref{alg:iHPEP}}
\emph{(}with $\lambda_i^k \equiv \lambda_i = \sigma/L_i$ for all $i =1, \dots, n$\emph{)} for solving \eqref{eq:Smip}.
\end{proposition}
\begin{proof}
The proof follows the same pattern as that of Proposition~\ref{pro:error-FB}. For convenience, we present a sketch.
From \eqref{eq:proxTFB_i} and some simple algebraic manipulations, we obtain the inclusion in \eqref{relationTFB}.
To proof the inequality, note first that, from the identity in \eqref{eq:proxTFB_i}, we have 
$ \lambda_i y_i^k + x_i^k  - (G_i \wti z^{\,k} + \lambda_i \wti w_i^k) = \lambda_i (F_i(x_i^k) - F_i(\overline z_i^k)) $, which in turn yields the desired result as a consequence of the $ L_i $-Lipschitz continuity of $ F_i(\cdot) $ and the definition of
$ \lambda_i $. The final claim of the proposition follows directly from \eqref{relationTFB} and the definitions of Algorithms~\ref{alg:iHPEP-Tseng-FB} and~\ref{alg:iHPEP}.
\end{proof}

\mgap

\brema \lab{rem:seifu02}
\emph{
In view of Proposition \ref{pro:error-TFB} and Corollary \ref{cor:airflo}, analogously to Remark \ref{rem:seifu}, from now on in this subsection 
we will assume that Algorithm \ref{alg:iHPEP-Tseng-FB} doesn't stop.
}
\erema

\mgap

\brema \lab{rem:1043}
\emph{
Similarly to \eqref{eq:def.up} and \eqref{eq:def.L}, we now define
\begin{align} \lab{eq:def.upb}
 \underline \lambda \coloneqq \dfrac{\sigma}{\overline L}\quad \mbox{and} \quad
 \overline \lambda \coloneqq \dfrac{\sigma}{\underline L},
\end{align}
where
\begin{align} \lab{eq:def.Lb}
\overline L \coloneqq \displaystyle \max_{i=1,\dots, n} L_i \quad \mbox{and} \quad
\underline L \coloneqq \displaystyle \min_{i=1,\dots, n} L_i.
\end{align}
Recall that, in this subsection, $L_i$ $(1\leq i\leq n)$ denotes the Lipschitz constant of $F_i(\cdot)$ 
(see \eqref{ineq:lipschitz}).
}
\erema

\mgap

\begin{theorem}[Strong convergence and iteration-complexity of Algorithm \ref{alg:iHPEP-Tseng-FB}]\label{theo:TFBComplexity}
Consider the sequences evolved by \emph{Algorithm \ref{alg:iHPEP-Tseng-FB}}, let $d_{0,\gamma}$ denote the distance of $p^0$ to extended-solution set $\Sa_e \neq \emptyset$, let $\overline L, \underline L > 0$ be as in \eqref{eq:def.Lb} and let $\Omega_{(\alpha, \beta)} > 0$ be as in \eqref{eq:def.omega}. Suppose assumptions \emph{B1--B4} and
\emph{\ref{assu:ab}} are valid (see also \emph{Remarks \ref{rem:adapter} and \ref{rem:1042}}), and that the (Lipschitz) condition 
\eqref{ineq:lipschitz} also holds on $F_i(\cdot)$ ($1\leq i\leq n$).
Then the following statements are true:
\begin{enumerate}[label = \emph{(\alph*)}]
\item \emph{(}Strong convergence\emph{)} All statements of \emph{Theorem \ref{th:sconv}} also hold for \emph{Algorithm \ref{alg:iHPEP-Tseng-FB}}
\item \emph{(}Iteration-complexity\emph{)} For all $k \geq 0$, there exists $\ell \in \{0,1, \ldots,k\}$ such that
\begin{align*}
\begin{cases}
y_i^\ell \in (F_i + B_i)(x_i^\ell), \quad i =1, \ldots ,n, & \\[3mm]
\Norm{\sum_{i=1}^{n-1}G_i^\ast y_i^\ell + y_n^\ell} \leq  
\dfrac{n \Big(\displaystyle \max_{i =1, \ldots,n}\normq{G_i}\Big) 
\max\left\{\dfrac{\overline L}{\sigma}, \dfrac{\sigma}{\underline L}\right\}d_{0,\gamma}}{\sqrt{k+1}} 
\left( \dfrac{4\sqrt{\gamma} \max\{1, 1/\gamma\}\sqrt{\Omega_{(\alpha, \beta)}}}{1-\sigma^2}\right), 
&\\[4mm]
\norm{x_i^\ell - G_i x_n^\ell} \leq \dfrac{n \Big(\displaystyle \max_{i =1, \ldots,n}\normq{G_i}\Big) 
\max\left\{\dfrac{\overline L}{\sigma}, \dfrac{\sigma}{\underline L}\right\} d_{0,\gamma}}{\sqrt{k+1}} 
\left( \dfrac{4 \max \{1, 1/\gamma\}\sqrt{\Omega_{(\alpha, \beta)}}}{1-\sigma^2}\right), &\\[5mm]
\hspace{10cm} i=1,\ldots, n-1.
\end{cases}
\end{align*}
\end{enumerate}
\end{theorem}
\begin{proof}
The proof follows directly from Proposition \ref{pro:error-TFB}, Remarks \ref{rem:adapter}, \ref{rem:seifu02} and \ref{rem:1043}, and Theorems \ref{th:sconv}
and \ref{theo:ite-com}.
\end{proof}

\section{Conclusions} \lab{sec:conc}
In this paper, we proposed and analyzed an inertial inexact projective splitting (PS) algorithm with strong convergence guarantees for solving multi-term monotone inclusion problems. The proposed framework offers greater flexibility while preserving key theoretical properties. Our main contribution lies in establishing the strong convergence of the iterates to the best approximation of the initial guess in the extended-solution set. To the best of our knowledge, this is the first instance of an inertial (relative-error) inexact PS-type algorithm with such a guarantee. Additionally, we introduced two algorithmic variants incorporating forward-backward and forward-backward-forward (Tseng) steps, further expanding the applicability of the PS method.
Inertial methods, which trace back to the celebrated heavy-ball and Nesterov’s accelerated gradient algorithms, have gained significant attention in the context of proximal-type algorithms since the seminal work~\cite{AA01}. As the introduction of inertial variants of known methods is often associated with improved convergence behavior in practice (see, e.g.,~\cite{Alv+20}), we believe that the inertial algorithms proposed in this paper have strong potential to outperform their non-inertial counterparts in terms of practical efficiency.

{\small
\section*{Acknowledgments}
The work of M.~M.~A. was partially supported by CNPq grant 308036/2021-2 and 
Fundação de Amparo a Pesquisa e Inovação do Estado de Santa Catarina (FAPESC), Edital 21/2024, Grant 2024TR002238.
The work of J.~E.~N.~C. was partially supported by CAPES 
scholarship 88887.829718/2023-00.
The work of R.~T.~M. was partially supported by  BASAL fund FB210005 for center of 
excellence from ANID-Chile.
}

\appendix

\section{Well-definiteness of Algorithm \ref{alg:iHPEP}} \lab{app:well}

The following proposition guarantees the well-definiteness of Algorithm \ref{alg:iHPEP}. 
Note that the definition of next iterate $(z^{k+1}, w_1^{k+1}, \ldots , w_{n-1}^{k+1})$ as in \eqref{eq:updt01} depends on the nonemptyness of the closed and convex set $H_k \cap W_k$. Also, recall we are assuming that the extended-solution set $\Sa_e$ is nonempty.

\mgap

\bprop \label{pro:inclutionset}
At every iteration $k\geq 0$, if \emph{Algorithm \ref{alg:iHPEP}} doesn't stop, then
the following conditions hold: $\Sa_e\subseteq H_k \cap W_k$, the next iterate $(z^{k+1}, w_1^{k+1}, \ldots , w_{n-1}^{k+1})$ as in \eqref{eq:updt01} is well-defined, and $\Sa_e \subseteq W_{k+1}$. 
\eprop
\bproo
Let's proceed by induction on $k\geq 0$. Let $k = 0$ and assume that Algorithm \ref{alg:iHPEP} doesn't stop. In this case, the half-spaces $H_0$ and $W_0$, as in \eqref{def:Hk-Wk}, can be defined and by Lemma \ref{lem:grad.sep}\emph{\ref{lem:grad.sep.ter}}, combined with the definition of $H_0$, we have 
$\Sa_e\subseteq H_0$. Since, trivially, $W_0 = \HHH$, then it follows that 
$\Sa_e\subseteq H_0\cap W_0$. As the extended-solution set $\Sa_e$ is nonempty, it follows that $H_0\cap W_0$ is also nonempty; obviously, it is also closed and convex. Hence, the next iterate 
$p^1 \coloneqq (z^1, w_1^1, \ldots , w_{n-1}^1) = P_{H_0\cap W_0}(p^0)$  is well-defined 
and (by Eq. \eqref{eq:pcr} in Lemma \ref{lm:proj}) we have 
$\inner{p^0 - p^1}{p - p^1}_\gamma \leq 0$ for all $p\in H_0\cap W_0$. In particular, the latter inequality also holds for all $p\in \Sa_e$, because $\Sa_e\subseteq H_0\cap W_0$, and this means exactly 
$\Sa_e \subseteq W_1$. We then conclude that the proposition holds for $k = 0$.

Suppose now it holds for some $k\geq 0$ and let's proof it for $k + 1$, assuming that we reached iteration 
$k +1$, that is, that Algorithm \ref{alg:iHPEP} doesn't stop at iteration $k\geq 0$: $\Sa_e\subseteq H_k \cap W_k$, the next iterate $(z^{k+1}, w_1^{k+1}, \ldots , w_{n-1}^{k+1})$ as in \eqref{eq:updt01} is well-defined, and $\Sa_e \subseteq W_{k+1}$. Then, if Algorithm \ref{alg:iHPEP} doesn't stop at iteration $k + 1$, it follows that
$\Sa_e \subseteq H_{k+1}$ (by Lemma \ref{lem:grad.sep}\emph{\ref{lem:grad.sep.ter}} and the definition of 
$H_{k+1}$). So, $\Sa_e \subseteq H_{k+1}\cap W_{k+1}$ and since $\Sa_e$ is nonempty, and $H_{k+1}\cap W_{k+1}$ is closed and convex, we have that 
$p^{k+2} \coloneqq (z^{k+2}, w_1^{k+2}, \ldots , w_{n-1}^{k+2}) = P_{H_{k+1}\cap W_{k+1}}(p^0)$
is well-defined and $\inner{p^0 - p^{k+2}}{p - p^{k+2}}_\gamma \leq 0$ for all $p \in H_{k+1}\cap W_{k+1}$. 
Using the inclusion $\Sa_e \subseteq H_{k+1}\cap W_{k+1}$, we have that the latter inequality also holds for all $p\in \Sa_e$, which is equivalent to say that $\Sa_e\subseteq W_{k+2}$. This concludes the induction process and the proof of the proposition.
\eproo

\section{Auxiliary results}
\lab{app:ar}

Throughout this section, $\HH$ and $\GG$ denote real Hilbert spaces.  The next result appeared 
in \cite{Alv22}, as a generalization of a previous result due to H. H. Bauschke:

\begin{lemma} \label{lem:MA}
Let $T \colon \HH \tos \HH$ be a maximal monotone operator and $V$ be a closed linear subspace of $\HH$. 
Let $\seq{(x^k, u^k)}$ be a sequence in $\Gra T^{[\varepsilon^k]}$ converging weakly to an element
$(\overline x, \overline u)$ in  $\HH \times \HH$ and assume that $\varepsilon^k \to 0$. If $P_{V^\perp} x^k \to 0$ and 
$P_V u^k \to 0$ \emph{(}$P_V$ and $P_{V^\perp}$ denote the orthogonal projection onto $V$ and $V^\perp$, respectively\emph{)}, 
then $(\overline x, \overline u) \in (\Gra T) \cap (V \times V^\perp)$ and $\inner{x^k}{u^k} \to \inner{\overline x}{\overline u} = 0$.
\end{lemma}

\mgap

The  lemma below is a generalization of \cite[Proposition 2.4]{ACS14}, and its proof, which we present here, follows the same ideas of that proposition's proof.

\mgap

\begin{lemma}\label{lem:MA2}
Let $A \colon \HH \tos \HH$ and $B \colon \GG \tos \GG$ be maximal monotone operators and let $G \colon \Ha \to \GG$ be a bounded linear operator. 
Assume that $\seq{(r^k, a^k)}$ and $\seq{(s^k, b^k)}$ are sequences in $\Gra A^{[\varepsilon^{a, k}]}$ and
$\Gra B^{[\varepsilon^{b, k}]}$, respectively, such that $r^k \wto \overline{r}\in \HH$ and
$b^k \wto \overline{b}\in \GG$, where $\varepsilon^{a, k}\to 0$ and $\varepsilon^{b, k}\to 0$.
If $Gr^k - s^k \to 0$ and $a^k + G^\ast b^k \to 0$, then
$(\overline{r}, -G^\ast \overline{b})\in \Gra A$ and $(G\overline{r}, \overline{b})\in \Gra B$.
\end{lemma} 
\begin{proof}
We apply Lemma \ref{lem:MA} (in the Hilbert space $\HH\times \GG$). For this, let
\begin{align} \lab{eq:def.tv}
\begin{aligned}
&T \coloneqq A\times B: \HH\times \GG\tos \HH\times \GG, \quad T(z, v) = A(z)\times B(v)\quad \mbox{and}\\[2mm]
&V \coloneqq \set{(z, v)\in \HH\times \GG \mid v = Gz }.
\end{aligned}
\end{align}
Note that $T$ is maximal monotone, and $V$ is a closed linear subspace of $\HH\times \GG$. Using the assumptions on the sequences $\seq{(r^k, a^k)}$, $\seq{(s^k, b^k)}$, the definition of $T(\cdot)$ as in \eqref{eq:def.tv} and Lemma 
\ref{lem:prope}, we obtain
\begin{align} \lab{eq:redpen}
\begin{aligned}
&(x^k, u^k)\in \Gra T^{[\varepsilon^k]}\quad \mbox{for all}\quad k\geq 0\quad \mbox{where}\\[2mm]
& x^k \coloneqq (r^k, s^k), \quad u^k\coloneqq (a^k, b^k)\quad \mbox{and}\quad \varepsilon^k\coloneqq 
\varepsilon^{a, k} + \varepsilon^{b, k}.
\end{aligned}
\end{align}
Assume now that 
\begin{align}\lab{eq:redpen02}
 Gr^k - s^k \to 0\quad  \mbox{and}\quad a^k + G^\ast b^k \to 0. 
 \end{align}
 Then, in view of the assumptions that 
$r^k \wto \overline{r}$ and $b^k \wto \overline{b}$, and the (weak) continuity of the linear operators $G$ and $G^\ast$,
we also have $s^k\wto G\,\overline r$ and $a^k\wto -G^\ast\,\overline b$. Hence, also using the assumptions that
both $\seq{\varepsilon^{a, k}}$ and $\seq{\varepsilon^{b, k}}$ converge to zero, we find
\begin{align}\lab{eq:redpen03}
x^k\wto (\overline r, G\,\overline r)\eqqcolon \overline x, \quad u^k\wto (-G^\ast\overline b, \overline b)\eqqcolon \overline u
\quad \mbox{and} \quad \varepsilon^k \to 0,
\end{align}
where $x^k$, $u^k$ and $\varepsilon^k$ are as in \eqref{eq:redpen}.

Now, let us check the conditions $x^k - P_V x^k \to 0$ and 
$P_V u^k \to 0$. From \cite[Lemma 3.1]{ACS14} we know that, for all $w = (z, v)\in \HH\times \GG$,
\begin{align*}
\begin{aligned}
& P_V w = \left( (I + G^\ast G)^{-1}(z + G^\ast v), G(I + G^\ast G)^{-1}(z + G^\ast v) \right),\\[2mm]
& P_{V^\perp} w = \left( G^\ast(I + G G^\ast))^{-1}(Gz - v),-(I + G G^\ast)^{-1}(Gz - v) \right),
\end{aligned}
\end{align*}
which in turn combined with \eqref{eq:redpen02} yields
\begin{align}\lab{eq:redpen04}
  P_{V^\perp} x^k \to 0 \quad \mbox{and} \quad P_V u^k \to 0,
\end{align}
where $x^k$ and $u^k$ are as in \eqref{eq:redpen}. Hence, from \eqref{eq:redpen}, \eqref{eq:redpen03}, \eqref{eq:redpen04}
and Lemma \ref{lem:MA} we conclude, in particular, that $(\overline x, \overline u)\in \Gra T$, i.e, 
$\overline u\in T(\overline x)$, which gives -- see \eqref{eq:redpen03} --
\begin{align*}
(-G^\ast\overline b, \overline b) \in (A \times B)(\overline r, G\,\overline r)
\end{align*}
and so $-G^\ast \overline b\in A(\overline r)$ and $\overline b\in B(G\overline r)$, concluding the proof of the lemma.
\end{proof}

\mgap

In the next corollary, we assume that $\HH_i$ ($1\leq i\leq n$) are real Hilbert spaces with $\HH_n = \HH_0$.

\mgap

\bcoro \lab{cor:proof}
For $i = 1, \ldots, n$, let $T_i \colon \Ha_i \tos \Ha_i$ be a maximal monotone operator, let $\seq{(x_i^k, y_i^k)}$ be a sequence in $\Gra T_i^{[\varepsilon_i^k]}$, where $\varepsilon_i^k \to 0$, and assume $x_n^k\wto \overline z\in \HH_0$ and 
$y_i^k \wto \overline w_i\in \HH_i$ \emph{(}$1\leq i\leq n-1$\emph{)}. 
Also, for $i = 1, \ldots, n-1$, let $G_i \colon \Ha_0 \to  \Ha_i$  be a bounded linear operator.
If $G_i x_n^k - x_i^k \to 0$ \emph{(}$1\leq i\leq n-1$\emph{)} and $y_n^k + \sum_{i=1}^{n-1}\,G_i^* y_i^k \to 0$, then
\begin{align} \lab{eq:proof01}
\overline w_i \in T_i (G_i \overline z),\quad i = 1, \ldots, n-1,\quad
 -\sum_{i=1}^{n-1}\,G_i^* \overline w_i \in T_n(\overline z).
\end{align}
\ecoro
\begin{proof}
Let $\HH \coloneqq \HH_0$, $\GG \coloneqq \HH_1 \times \dots \times \HH_{n-1}$ (endowed with the usual (Cartesian) inner product) and define the maximal monotone operators $A\colon \HH \tos \HH$ and $B\colon \GG \tos \GG$ by
\begin{align*}
A \coloneqq T_n,\qquad B \coloneqq T_1\times \dots \times T_{n-1}.
\end{align*}
Also, define the bounded linear operator $G\colon \HH \to \GG$ by $G = (G_1, \ldots, G_{n-1})$, i.e,
\begin{align*}
Gz \coloneqq (G_1z, \ldots, G_{n-1}z)\quad \mbox{for all}\quad z\in \HH.
\end{align*}
Now, define the sequences
\begin{align*}
& a^k\coloneqq y_n^k,\quad \varepsilon^{a, k}\coloneqq \varepsilon_n^k,\quad r^k\coloneqq x_n^k,\\[2mm]
& b^k\coloneqq (y_1^k, \ldots, y_{n-1}^k),\quad \varepsilon^{b, k}\coloneqq \sum_{i=1}^{n-1}\,\varepsilon_i^k,
\quad s^k\coloneqq (x_1^k,\ldots, x_{n-1}^k).
\end{align*}

Using the above definitions and Lemma \ref{lem:prope}, we conclude that $\seq{(r^k, a^k)}$ and $\seq{(s^k, b^k)}$ are sequences in $\Gra A^{[\varepsilon^{a, k}]}$ and $\Gra B^{[\varepsilon^{b, k}]}$, respectively. Moreover, using the assumptions of the corollary we also have $\varepsilon^{a, k}\to 0$, 
$\varepsilon^{b, k}\to 0$ and
\begin{align*}
r^k\wto \overline z \eqqcolon \overline r \quad \mbox{and} \quad b^k \wto (\overline w_1, \ldots, \overline w_{n-1}) \eqqcolon \overline b.
\end{align*}

In order to apply Lemma \ref{lem:MA2}, it remains to verify the conditions $Gr^k - s^k \to 0$ and $a^k + G^\ast b^k \to 0$. To this end, note first that
\begin{align*}
Gr^k - s^k &= (G_1r^k, \ldots, G_{n-1}r^k) - (x_1^k,\ldots, x_{n-1}^k)\\
    & = (G_1x_n^k - x_1^k, \ldots, G_{n-1}x_n^k - x_{n-1}^k),
\end{align*}
and so, by the assumptions of the corollary, we get $Gr^k - s^k \to 0$. Using the fact that
$G^\ast (w_1, \ldots, w_{n-1}) = \sum_{i=1}^{n-1}G_i^\ast w_i$, for every $(w_1, \ldots , w_{n-1}) \in \GG$, we also obtain
\begin{align*}
 a^k + G^\ast b^k = y_n^k + \sum_{i=1}^{n-1}G_i^\ast y_i^k,
\end{align*}
which then converges to zero, also in view of the assumption that $y_n^k + \sum_{i=1}^{n-1}G_i^\ast y_i^k \to 0$.

\mgap

Therefore, by Lemma \ref{lem:MA2}, we conclude that $(\overline{r}, -G^\ast \overline{b})\in \Gra A$ and $(G\overline{r}, \overline{b})\in \Gra B$, which happens to be equivalent to \eqref{eq:proof01}.
\end{proof}



\begin{thebibliography}{10}

\bibitem{Alg+14}
M.~A. Alghamdi, A.~Alotaibi, P.~L. Combettes, and N.~Shahzad.
\newblock A primal-dual method of partial inverses for composite inclusions.
\newblock {\em Optim. Lett.}, 8(8):2271--2284, 2014.

\bibitem{ACS14}
A.~Alotaibi, P.~L. Combettes, and N.~Shahzad.
\newblock Solving coupled composite monotone inclusions by successive
  {F}ej\'{e}r approximations of their {K}uhn-{T}ucker set.
\newblock {\em SIAM J. Optim.}, 24(4):2076--2095, 2014.

\bibitem{AA01}
F.~Alvarez and H.~Attouch.
\newblock An inertial proximal method for maximal monotone operators via
  discretization of a nonlinear oscillator with damping.
\newblock {\em Set-Valued Anal.}, 9(1-2):3--11, 2001.

\bibitem{Alv22}
M.~M. Alves.
\newblock Another proof and a generalization of a theorem of {H}. {H}.
  {B}auschke on monotone operators.
\newblock {\em Optimization}, 71(1):91--96, 2022.

\bibitem{Alv25}
M.~M. Alves.
\newblock Projective splitting with backward, half-forward and proximal-newton
  steps.
\newblock {\em J. Convex Anal.}, 32(4):1199--1226, 2025.

\bibitem{ACM24}
M.~M. Alves, J.~E.~Navarro Caballero, and R.~T. Marcavillaca.
\newblock A strongly convergent inertial inexact proximal-point algorithm for
  monotone inclusions with applications to variational inequalities.
\newblock \emph{arXiv:2407.03485 [math.OC]}, 2024.

\bibitem{Alv+20}
M.~M. Alves, J.~Eckstein, M.~Geremia, and J.G. Melo.
\newblock Relative-error inertial-relaxed inexact versions of
  {D}ouglas-{R}achford and {ADMM} splitting algorithms.
\newblock {\em Comput. Optim. Appl.}, 75(2):389--422, 2020.

\bibitem{AGM24}
M.~M. Alves, M.~Geremia, and R.~T. Marcavillaca.
\newblock A relative-error inertial-relaxed inexact projective splitting
  algorithm.
\newblock {\em J. Convex Anal.}, 31(1):1--24, 2024.

\bibitem{AM20}
M.~M. Alves and R.~T. Marcavillaca.
\newblock On inexact relative-error hybrid proximal extragradient,
  forward-backward and {T}seng's modified forward-backward methods with
  inertial effects.
\newblock {\em Set-Valued Var. Anal.}, 28(2):301--325, 2020.

\bibitem{AC20}
H.~Attouch and A.~Cabot.
\newblock Convergence of a relaxed inertial proximal algorithm for maximally
  monotone operators.
\newblock {\em Math. Program.}, 184(1-2, Ser. A):243--287, 2020.

\bibitem{AP19}
H.~Attouch and J.~Peypouquet.
\newblock Convergence of inertial dynamics and proximal algorithms governed by
  maximally monotone operators.
\newblock {\em Math. Program.}, 174(1-2, Ser. B):391--432, 2019.

\bibitem{BC01}
H.~H. Bauschke and P.~L. Combettes.
\newblock A weak-to-strong convergence principle for {F}ej{\'{e}r}-monotone
  methods in {H}ilbert spaces.
\newblock {\em Math. Oper. Res.}, 26(2):248--264, 2001.

\bibitem{BC11}
H.~H. Bauschke and P.~L. Combettes.
\newblock {\em Convex analysis and monotone operator theory in {H}ilbert
  spaces}.
\newblock CMS Books in Mathematics/Ouvrages de Math\'ematiques de la SMC.
  Springer, New York, 2011.

\bibitem{BC20}
M.~N. B\`ui and P.~L. Combettes.
\newblock Warped proximal iterations for monotone inclusions.
\newblock {\em J. Math. Anal. Appl.}, 491(1):124315, 21, 2020.

\bibitem{BC22}
M.~N. B\`ui and P.~L. Combettes.
\newblock Multivariate monotone inclusions in saddle form.
\newblock {\em Math. Oper. Res.}, 47(2):1082--1109, 2022.

\bibitem{BSS99}
R.~S. Burachik, C.~A. Sagastiz{\'a}bal, and B.~F. Svaiter.
\newblock {$\epsilon$}-enlargements of maximal monotone operators: theory and
  applications.
\newblock In {\em Reformulation: nonsmooth, piecewise smooth, semismooth and
  smoothing methods ({L}ausanne, 1997)}, volume~22 of {\em Appl. Optim.}, pages
  25--43. Kluwer Acad. Publ., Dordrecht, 1999.

\bibitem{CE18}
P.~L. Combettes and J.~Eckstein.
\newblock Asynchronous block-iterative primal-dual decomposition methods for
  monotone inclusions.
\newblock {\em Math. Program.}, 168(1-2, Ser. B):645--672, 2018.

\bibitem{CG17}
P.~L. Combettes and L.~E. Glaudin.
\newblock Quasi-nonexpansive iterations on the affine hull of orbits: from
  {M}ann's mean value algorithm to inertial methods.
\newblock {\em SIAM J. Optim.}, 27(4):2356--2380, 2017.

\bibitem{ES08}
J.~Eckstein and B.~F. Svaiter.
\newblock A family of projective splitting methods for the sum of two maximal
  monotone operators.
\newblock {\em Math. Program.}, 111(1-2, Ser. B):173--199, 2008.

\bibitem{ES09}
J.~Eckstein and B.~F. Svaiter.
\newblock General projective splitting methods for sums of maximal monotone
  operators.
\newblock {\em SIAM J. Control Optim.}, 48(2):787--811, 2009.

\bibitem{Gul91}
O.~G\"{u}ler.
\newblock On the convergence of the proximal point algorithm for convex
  minimization.
\newblock {\em SIAM J. Control Optim.}, 29(2):403--419, 1991.

\bibitem{EJ19}
P.~R. Johnstone and J.~Eckstein.
\newblock Convergence rates for projective splitting.
\newblock {\em SIAM J. Optim.}, 29(3):1931--1957, 2019.

\bibitem{JE21}
P.~R. Johnstone and J.~Eckstein.
\newblock Single-forward-step projective splitting: exploiting cocoercivity.
\newblock {\em Comput. Optim. Appl.}, 78(1):125--166, 2021.

\bibitem{JE22}
P.~R. Johnstone and J.~Eckstein.
\newblock Projective splitting with forward steps.
\newblock {\em Math. Program.}, 191(2, Ser. A):631--670, 2022.

\bibitem{LM79}
P.-L. Lions and B.~Mercier.
\newblock Splitting algorithms for the sum of two nonlinear operators.
\newblock {\em SIAM J. Numer. Anal.}, 16(6):964--979, 1979.

\bibitem{Mac25}
M.~P. Machado.
\newblock An inertial projective splitting method for the sum of two maximal
  monotone operators.
\newblock {\em Comp. Appl. Math.}, 44(84), 2025.

\bibitem{MS23}
M.~P. Machado and M.~R. Sicre.
\newblock A projective splitting method for monotone inclusions:
  Iteration-complexity and application to composite optimization.
\newblock {\em J. Optim. Theory Appl.}, 198:552–587, 2023.

\bibitem{MS10}
R.~D.~C. Monteiro and B.~F. Svaiter.
\newblock On the complexity of the hybrid proximal extragradient method for the
  iterates and the ergodic mean.
\newblock {\em SIAM J. Optim.}, 20(6):2755--2787, 2010.

\bibitem{MS11}
R.~D.~C. Monteiro and B.~F. Svaiter.
\newblock Complexity of variants of {T}seng's modified {F}-{B} splitting and
  {K}orpelevich's methods for hemivariational inequalities with applications to
  saddle-point and convex optimization problems.
\newblock {\em SIAM J. Optim.}, 21(4):1688--1720, 2011.

\bibitem{Pas79}
G.~B. Passty.
\newblock Ergodic convergence to a zero of the sum of monotone operators in
  {H}ilbert space.
\newblock {\em J. Math. Anal. Appl.}, 72(2):383--390, 1979.

\bibitem{Roc76}
R.~T. Rockafellar.
\newblock Monotone operators and the proximal point algorithm.
\newblock {\em SIAM J. Control Optim.}, 14(5):877--898, 1976.

\bibitem{SS99b}
M.~V. Solodov and B.~F. Svaiter.
\newblock A hybrid approximate extragradient-proximal point algorithm using the
  enlargement of a maximal monotone operator.
\newblock {\em Set-Valued Anal.}, 7(4):323--345, 1999.

\bibitem{SS00}
M.~V. Solodov and B.~F. Svaiter.
\newblock Forcing strong convergence of proximal point iterations in a
  {H}ilbert space.
\newblock {\em Math. Program.}, 87(1, Ser. A):189--202, 2000.

\bibitem{Sva14}
B.~F. Svaiter.
\newblock A class of {F}ej\'er convergent algorithms, approximate resolvents
  and the hybrid proximal-extragradient method.
\newblock {\em J. Optim. Theory Appl.}, 162(1):133--153, 2014.

\bibitem{Tse00}
P.~Tseng.
\newblock A modified forward-backward splitting method for maximal monotone
  mappings.
\newblock {\em SIAM J. Control Optim.}, 38(2):431--446 (electronic), 2000.

\end{thebibliography}


\def\cprime{$'$} \def\cprime{$'$}

\end{document}